\documentclass[12pt,a4paper]{amsart}
\usepackage{amsfonts}
\usepackage{amssymb}
\usepackage{amsmath}
\usepackage{amsthm}
\usepackage{cite}
\usepackage{enumitem}
\usepackage{tikz}
\usepackage{subfigure}
\usepackage{epstopdf}
\usepackage{graphicx}

\theoremstyle{plain}
\newtheorem{thm}{Theorem}
\newtheorem{defn}{Definition}
 
\newtheorem{prop}{Proposition}
\newtheorem{cor}{Corollary}
\newtheorem{bem}{Remark}

\newcommand{\vecII}[2]{
\ensuremath{
\begin{pmatrix}
#1 \\ #2 \\
\end{pmatrix}}}
\newcommand{\matII}[4]{
\ensuremath{ 
\begin{pmatrix}  
#1 & #2 \\
#3 & #4 \\
\end{pmatrix}}}

\providecommand{\N}{\mathbb{N}}
\providecommand{\R}{\mathbb{R}}
\providecommand{\Z}{\mathbb{Z}}

\DeclareMathOperator{\cn}{cn}
\DeclareMathOperator{\sn}{sn}
\DeclareMathOperator{\dn}{dn}

\DeclareMathOperator{\sign}{sign}

\renewcommand{\qed}{\hfill $\Box$}

\setitemize{itemsep=+2pt}
\setenumerate{itemsep=+2pt}

\begin{document}

\allowdisplaybreaks

\title{Boundary value problems for Willmore curves in $\R^2$}

\author{Rainer Mandel}
\address{R. Mandel \hfill\break
Scuola Normale Superiore di Pisa \hfill\break
I - Pisa, Italy}
\email{Rainer.Mandel@sns.it}
\date{}

\subjclass[2000]{Primary: }
\keywords{}

\begin{abstract}
  In this paper the Navier problem and the Dirichlet problem for Willmore curves in $\R^2$ is solved. 
\end{abstract}

\maketitle

 \section{Introduction}
 
 \parindent0mm
 
  The Willmore functional of a given smooth regular curve $\gamma\in C^\infty([0,1],\R^n)$ is given by
  $$
    W(\gamma) := \frac{1}{2}\int_\gamma \kappa_\gamma^2   
  $$
  where $\kappa_\gamma$ denotes the signed curvature function of $\gamma$.
  The value $W(\gamma)$ can be interpreted as a bending energy of the curve $\gamma$ which one usually tends
  to minimize in a suitable class of admissible curves where the latter is chosen according to the mathematical
  or real-world problem one is interested in. One motivation for considering boundary value
  problems for Willmore curves originates from the industrial fabrication process called wire bonding
  \cite{Koo_ideal_bent_contours}. Here, engineers aim at interconnecting small semiconductor devices via wire
  bonds in such a way that energy losses within the wire are smallest possible. Typically the initial position
  and direction as well as the end position and direction of the curve are prescribed by the arrangement of the
  semiconductor devices. Neglecting energy losses due to the length of a wire one ends up with studying the
  Dirichlet problem for Willmore curves in $\R^3$. In the easier two-dimensional case the Dirichlet problem
  consists of finding a regular curve $\gamma\in C^\infty([0,1],\R^2)$ satisfying
  \begin{align} \label{Gl Dirichlet BVP}
    W'(\gamma)=0,\qquad 
    \gamma(0)=A,\; \gamma(1)=B,\qquad
    \frac{\gamma'(0)}{\|\gamma'(0)\|}= \vecII{\cos(\theta_1)}{\sin(\theta_1)},\;
    \frac{\gamma'(1)}{\|\gamma'(1)\|}= \vecII{\cos(\theta_2)}{\sin(\theta_2)}
  \end{align}
  for given boundary data $A,B\in\R^2$ and $\theta_1,\theta_2\in\R$. Here we used the shorthand notation
  $W'(\gamma)=0$ in order to indicate that $\gamma$ is a Willmore curve which, by
  definition, is a critical point of the Willmore functional. 
  Our main result, however, concerns the Navier problem for Willmore curves in $\R^2$ where one looks for
  regular curves $\gamma\in C^\infty([0,1],\R^2)$ satisfying
  \begin{align} \label{Gl Navier BVP}
    W'(\gamma)=0,\qquad \gamma(0)=A,\;\gamma(1)=B,\qquad
    \kappa_\gamma(0)=\kappa_1,\;\kappa_\gamma(1)=\kappa_2
  \end{align}
  for given $A,B\in\R^2$ and $\kappa_1,\kappa_2\in\R$. 
  In the special case $\kappa_1=\kappa_2=:\kappa$ the
  boundary value problem  \eqref{Gl Navier BVP} was investigated by Deckelnick and Grunau
  \cite{DecGru_Boundary_value_problems} under the restrictive assumption that $\gamma$ is a smooth symmetric
  graph lying on one side of the straight line joining $A$ and $B$. More precisely 
  they assumed $A=(0,0),B=(1,0)$ as well as $\gamma(t)=(t,u(t))\,(t\in [0,1])$ for some positive
  symmetric function $u\in C^4(0,1)\cap C^2[0,1]$. They found that there is a positive number $M_0\approx 1.34380$ such
  that \eqref{Gl Navier BVP} has precisely two graph solutions for $0<|\kappa|<M_0$, precisely one graph solution for $|\kappa|\in\{0,M_0\}$ and no
  graph solutions otherwise, see Theorem~1 in \cite{DecGru_Boundary_value_problems}.
  In a subsequent paper \cite{DecGru_Stability_and_symmetry} the same authors proved stability and symmetry
  results for such Willmore graphs. Up to the author's knowledge no further results related to the Navier
  problem are known. The aim of our paper is to fill this gap in the literature by providing
  the complete solution theory both for the Navier problem and for the Dirichlet problem. Special attention
  will be paid to the symmetric Navier problem
  \begin{align} \label{Gl Navier BVP symm}
    W'(\gamma)=0,\qquad
    \gamma(0)=A,\; \gamma(1)=B,\qquad
    \kappa_\gamma(0)=\kappa,\; \kappa_\gamma(1)=\kappa 
  \end{align}
  which, as we will show in Theorem \ref{Thm Navier kappa1=kappa2}, admits a beautiful solution theory
  extending the results obtained in \cite{DecGru_Boundary_value_problems} which we described above. Our
  results concerning this boundary value problem not only address existence and multiplicity issues but also
  qualitative information about the solutions are obtained. In \cite{DecGru_Boundary_value_problems} the
  authors asked whether symmetric boundary conditions imply the symmetry of the solution. In order to analyse
  such interesting questions not only for Willmore graphs as in
  \cite{DecGru_Boundary_value_problems,DecGru_Stability_and_symmetry} but also for general Willmore curves we
  fix the notions of axially symmetric respectively pointwise symmetric regular curves which we believe to be natural.
   
  \begin{defn}
    Let $\gamma\in C^\infty([0,1],\R^2)$ be a regular curve of length $L>0$ and let $\hat\gamma$ denote its
    arc-length parametrization. Then $\gamma$ is said to be
    \begin{itemize}
      \item[(i)] axially symmetric if there is an orthogonal matrix $P\in O(2)$ satisfying $\det(P)=-1$
      such that $s\mapsto \hat\gamma(s)+P\hat\gamma(L-s)$ is constant on $[0,L]$,
      \item[(ii)] pointwise symmetric if $s\mapsto \hat\gamma(s)+\hat\gamma(L-s)$ is constant on $[0,L]$.
    \end{itemize}
  \end{defn}
   
   Let us shortly explain why these definitions are properly chosen. First, being axially or
   pointwise symmetric does not depend on the parametrization of the curve, which is a necessary condition for
   a geometrical property. Second, the definition from (i) implies
   $P(\hat\gamma(L)-\hat\gamma(0))=\hat\gamma(L)-\hat\gamma(0)$ so that $P$ realizes a
   reflection about the straight line from $\gamma(0)=\hat\gamma(0)$ to $\gamma(1)=\hat\gamma(L)$. Hence, (i)
   may be considered as a natural generalization of the notion of a symmetric function. Similarly, (ii) extends the notion of an odd
   function. Since our results describe the totality of Willmore curves in $\R^2$ we will need more refined
   ways of distinguishing them. One way of defining suitable types of Willmore curves is based on the observation
   (to be proved later) that the curvature function of the arc-length parametrization of such a curve
   is given by
   \begin{equation} \label{Gl kappa_ab}
     \kappa_{a,b}(s) = \sqrt{2}a \cn(as+b)   \qquad a\geq 0,\; -T/2\leq b<T/2.
   \end{equation}
   Here, the symbol $\cn$ denotes Jacobi's elliptic function with modulus $1/\sqrt{2}$ and periodicity $T>0$,
   cf. \cite{AbrSte_Handbook}, chapter 16. In view of this result it turns out that the following notion of a
   $(\sigma_1,\sigma_2,j)$-type solution provides a suitable classification of all Willmore curves except the
   straight line (which corresponds to the special case $a=0$ in \eqref{Gl kappa_ab}).
 
  \begin{defn} \label{Def sigma1sigma2j}
    Let $\gamma\in C^\infty([0,1],\R^2)$ be a Willmore curve of length $L>0$ such that the curvature
    function of its arc-length parametrization is $\kappa_{a,b}$ for $a>0$ and $b\in [-T/2,T/2)$. Then $\gamma$ is said to be
    of type $(\sigma_1,\sigma_2,j)$ if we have 
    $$
      \sign(b)=\sigma_1,\qquad
      \sign(b+aL)=\sigma_2,\qquad
      jT \leq aL<(j+1)T.
    $$
  \end{defn}
  
  Here, the sign function is given by $\sign(z):=1$ for $z\geq 0$ and $\sign(z)=-1$ for $z<0$. Roughly
  speaking, a Willmore curve is of type $(\sigma_1,\sigma_2,j)$ if its curvature function runs through the
  repetitive pattern given by the $\cn$-function more than $j$ but less than $j+1$ times and is increasing or
  decreasing at its initial point respectively end point according to $\sigma_1=\pm 1$ respectively
  $\sigma_2=\pm 1$. This particular shape of the curvature function of a Willmore curve clearly induces a
  special shape of the Willmore curve itself as the following examples show.
  
  \medskip

  \begin{figure}[!htb] 
    \centering
    \subfigure[$\sigma_1=1,\sigma_2=-1,j=1$
    ]{
      \includegraphics[scale=.3]{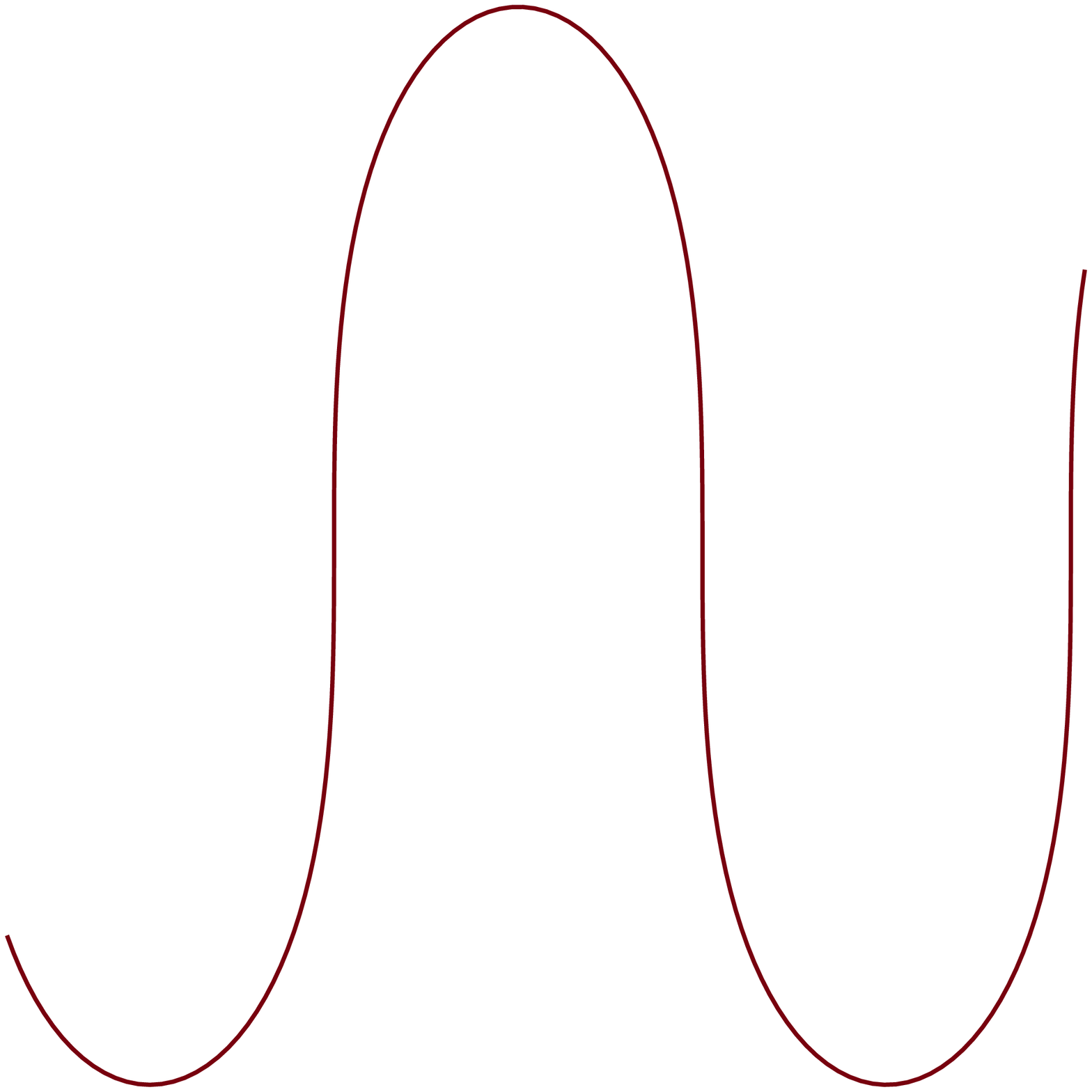}
    }
    \qquad\qquad\qquad
    \subfigure[$\sigma_1=-1,\sigma_2=-1,j=3$
    ]{
      \includegraphics[scale=.3]{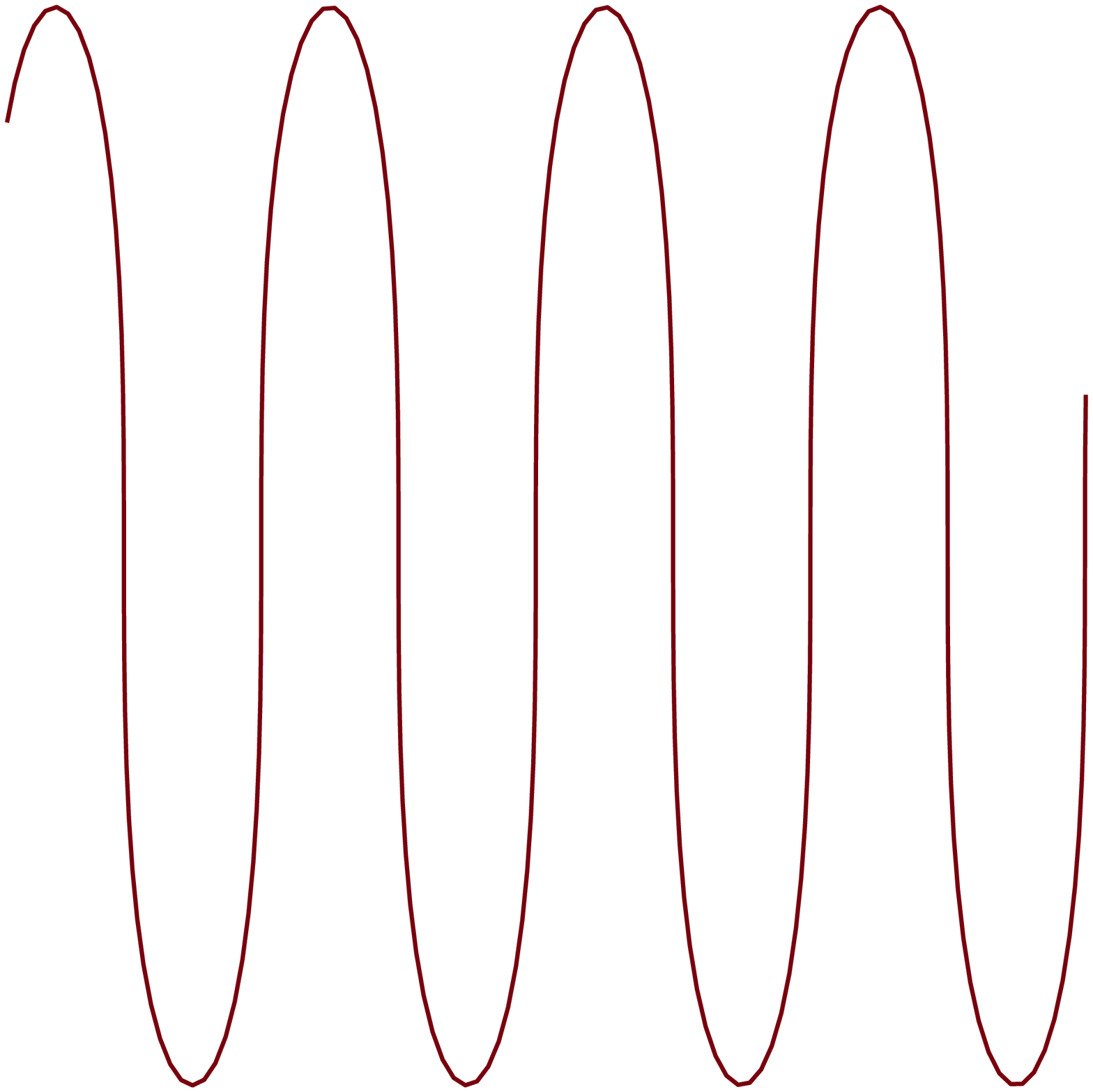}
    }
    \caption{Willmore curves of type $(\sigma_1,\sigma_2,j)$}
  \end{figure}
   
  \medskip 
  
  In order to describe the entire solution set $\mathcal{F}$ of the symmetric Navier problem \eqref{Gl Navier
  BVP symm} we introduce the solution classes $\mathcal{F}_0,\mathcal{F}_1,\ldots$ by setting
  $\mathcal{F}_k:= \mathcal{F}_k^+\cup\mathcal{F}_k^-$ and
   \begin{align} \label{Gl def Fk}
     \begin{aligned}
     \mathcal{F}_0^\pm
     &:= \Big\{ (\gamma,\kappa)\in C^\infty([0,1],\R^2)\times\R: \; \kappa=0,\;
     \gamma \text{ is the straight line through } A,B \\
     &\qquad\qquad\text{or } 0\leq \pm\kappa<\infty,\; \gamma \text{ is
     a } (\mp 1,\pm 1,0)-\text{type solution of \eqref{Gl Navier BVP symm}} \Big\}, \\
     \mathcal{F}_k^\pm 
     &:= \Big\{ (\gamma,\kappa)\in C^\infty([0,1],\R^2)\times\R : \; 0\leq \pm\kappa<\infty \text{ and } 
     \gamma \text{ is a solution of \eqref{Gl Navier BVP symm}} \\
     &\qquad\qquad \text{of type } (\pm 1,\mp 1,k-1) \text{ or }(\mp 1,\pm 1,k) \text{ or } (1,1,k) \text{ or
     }(-1,-1,k)  \Big\}.
     \end{aligned}
   \end{align}  
   In our main result Theorem \ref{Thm Navier kappa1=kappa2} we show amongst other things that every solution
   of the symmetric Navier problem \eqref{Gl Navier BVP symm} belongs to precisely one of these sets, in
   other words we show $\mathcal{F}= \bigcup_{k\in\N_0} \mathcal{F}_k$ where the union of the right hand side
   is disjoint. In this decomposition the sets $\mathcal{F}_1,\mathcal{F}_2,\ldots$ will turn out to be of
   similar ''shape'' whereas $\mathcal{F}_0$ looks different, see figure \ref{Fig 1}.

  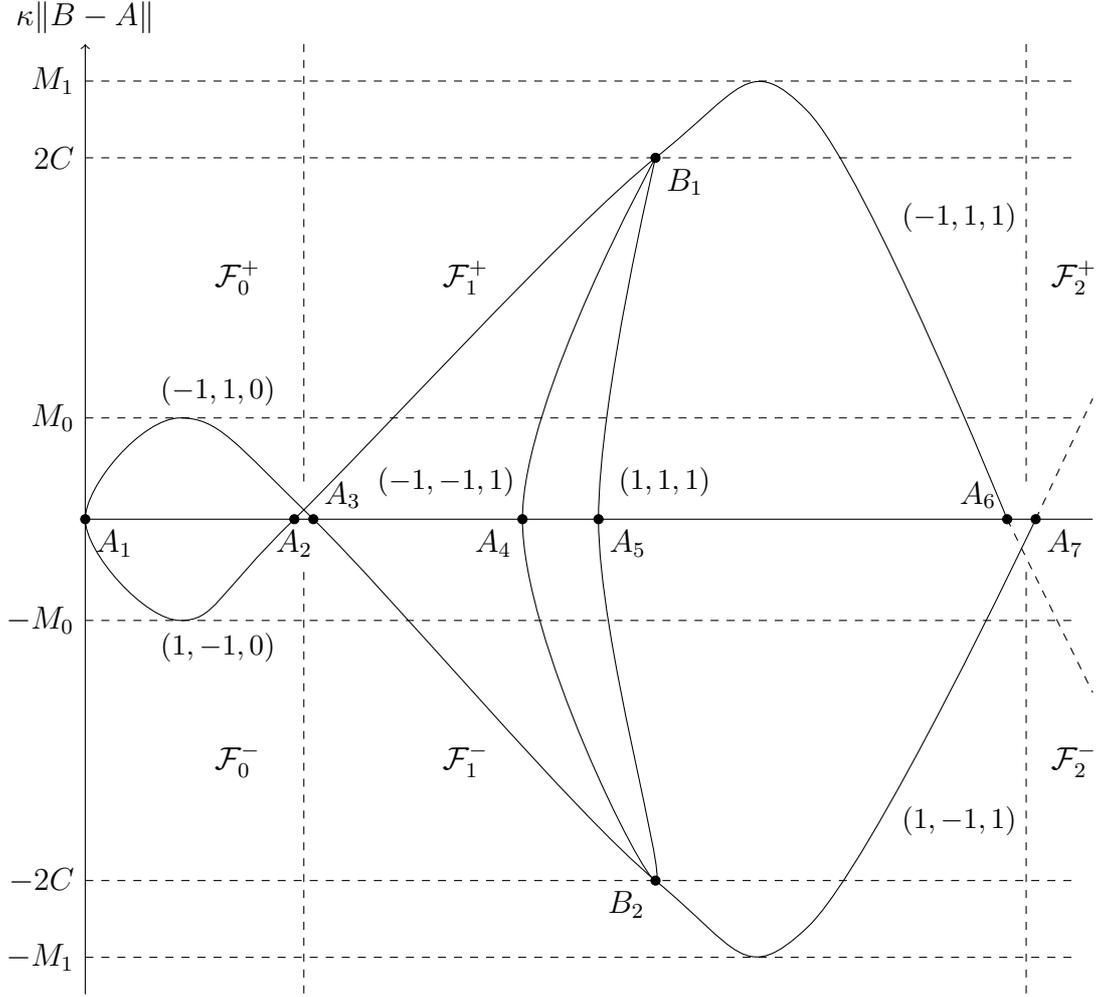
\begin{figure}[htp!] 
    \centering
    \begin{tikzpicture}[yscale=1, xscale=2.5]
	  \draw[-] (0,0) -- (5.3,0);
	  \draw[->] (0,-6.3) -- (0,6.3)  node[above] {$\kappa\|B-A\|$};
	  \draw[dashed] (0,1.343) node[left]{$M_0$} -- (5.2,1.343);
	  \draw[dashed] (0,-1.343) node[left]{$-M_0$} -- (5.2,-1.343);
	  \draw[dashed] (0,4.792) node[left]{$2C$} -- (5.2,4.792);
	  \draw[dashed] (0,-4.792) node[left]{$-2C$} -- (5.2,-4.792);
	  \draw[dashed] (0,5.81) node[left]{$M_1$} -- (5.2,5.81);  	   
	  \draw[dashed] (0,-5.81) node[left]{$-M_1$} -- (5.2,-5.81);
  	  \draw plot[smooth, tension=0.6] coordinates { 
  	  (4.85,0) (3.8,5.41) (3,4.792) (1.1,0) (0.5,-1.343)
  	  (0,0)  (0.5,1.343) (1.2,0) (3,-4.792) (3.8,-5.41) (5,0)};
  	  \draw plot[smooth, tension=0.6] coordinates { 
  	  (3,4.792) (2.3,0) (3,-4.792) (2.7,0) (3,4.792) };
  	  \draw[dashed] (5,0) -- (5.3,1.6);
  	  \draw[dashed] (4.85,0) -- (5.3,-2.3);
  	  \node at (0.7,1.7) {\small{$(-1,1,0)$}}; 
  	  \node at (0.7,-1.7) {\small{$(1,-1,0)$}};
  	  \node at (1.9,0.5) {\small{$(-1,-1,1)$}};
  	  \node at (3.05,0.5) {\small{$(1,1,1)$}};
  	  \node at (4.6,4) {\small{$(-1,1,1)$}};
  	  \node at (4.6,-4) {\small{$(1,-1,1)$}};
  	  \filldraw[black] (0,0) ellipse(1/2.5*0.06cm and 1/1*0.06cm);
  	  \node at (0,0)[below right]{$A_1$};
  	  \filldraw[black] (1.1,0) ellipse(1/2.5*0.06cm and 1/1*0.06cm);
  	  \node at (1.1,0)[below]{$A_2$};
  	  \filldraw[black] (1.2,0) ellipse(1/2.5*0.06cm and 1/1*0.06cm);
  	  \node at (1.2,0)[above right]{$A_3$};
  	  \filldraw[black] (2.3,0) ellipse(1/2.5*0.06cm and 1/1*0.06cm);
  	  \node at (2.3,0)[below left]{$A_4$};
  	  \filldraw[black] (2.7,0) ellipse(1/2.5*0.06cm and 1/1*0.06cm);
  	  \node at (2.7,0)[below right]{$A_5$};
  	  \filldraw[black] (4.85,0) ellipse(1/2.5*0.06cm and 1/1*0.06cm);
  	  \node at (4.85,0)[above left]{$A_6$};
  	  \filldraw[black] (5,0) ellipse(1/2.5*0.06cm and 1/1*0.06cm);
  	  \node at (5,0)[below right]{$A_7$};
  	  \filldraw[black] (3,4.792) ellipse(1/2.5*0.06cm and 1/1*0.06cm);
  	  \node at (3,4.792)[below right]{$B_1$};
  	  \filldraw[black] (3,-4.792) ellipse(1/2.5*0.06cm and 1/1*0.06cm);
  	  \node at (3,-4.792)[below left]{$B_2$};
  	  \draw[dashed] (1.15,0.5) -- (1.15,6.3);
  	  \draw[dashed] (1.15,-0.7) -- (1.15,-6.3);
  	  \draw[dashed] (4.95,0.5) -- (4.95,6.3);
  	  \draw[dashed] (4.95,-0.7) -- (4.95,-6.3);
  	  \node at (0.8,3.2){$\mathcal{F}_0^+$};
  	  \node at (0.8,-3.2){$\mathcal{F}_0^-$};
  	  \node at (2,3.2){$\mathcal{F}_1^+$};
  	  \node at (2,-3.2){$\mathcal{F}_1^-$};
  	  \node at (5.2,3.2){$\mathcal{F}_2^+$};
  	  \node at (5.2,-3.2){$\mathcal{F}_2^-$};
 	\end{tikzpicture}
 	\caption{An illustration of $\mathcal{F}_0$ and $\mathcal{F}_1$}
 	\label{Fig 1}  
  \end{figure}
   
   \newpage
      
   Due to space limitations we only illustrated the union of
   $\mathcal{F}_0=\mathcal{F}_0^+\cup\mathcal{F}_0^-$ and $\mathcal{F}_1=\mathcal{F}_1^+\cup\mathcal{F}_1^-$
   but the dashed branches emanating from $A_6,A_7$ are intended to indicate that
   $\mathcal{F}_1$ connects to the set $\mathcal{F}_2$ which has the same shape as $\mathcal{F}_1$ for
   $(2C,M_1)$ replaced by $(4C,M_2)$ and which in turn connects to $\mathcal{F}_3$ and so on. Here the real numbers $0<M_0<2C<M_1<4C<M_2<\ldots$ are given by 
   $$
     C:=  \int_0^1 \frac{4t^2}{\sqrt{1-t^4}}\,dt,
     \qquad
     M_k := \max_{z\in[0,1]}  2z \Big( kC+ 2\int_z^1 \frac{t^2}{\sqrt{1-t^4}}\,dt \Big). 
   $$
   The labels $(1,-1,0),(-1,1,0)$, etc. in figure \ref{Fig 1} indicate the type of the solution branch which
   is preserved until one of the bifurcation points $B_1,B_2$ or the trivial solution $A_1$ is reached. Let us
   remark that the horizontal line does not represent a parameter but only enables us to show seven different solutions $A_1,\ldots,A_7$
   of \eqref{Gl Navier BVP symm} for $\kappa=0$ as well as their interconnections in
   $C^\infty([0,1],\R^2)\times\R$. The solutions $A_1,A_2,A_3$ belong to
   $\mathcal{F}_0$ whereas $A_4,A_5,A_6,A_7$ belong to $\mathcal{F}_1$. The above picture indicates that the
   solution set $\mathcal{F}$ is pathconnected in any reasonable topology.
   Intending not to overload this paper we do not give a formal proof of this
   statement but only mention that a piecewise parametrization of the solution sets can be realized using the parameter
   $z\in [0,1]$ from the proof of Theorem~\ref{Thm Navier kappa1=kappa2}. Our result leading to
   figure~\ref{Fig 1} reads as follows.
   
  \begin{thm} \label{Thm Navier kappa1=kappa2}
    Let $A,B\in\R^2,A\neq B$ and $\kappa\in\R$. Then the following statements concerning the symmetric Navier
    problem \eqref{Gl Navier BVP symm}  hold true:
    \begin{itemize}
      \item[(i)] There are 
        \begin{itemize}
          \item[(a)] precisely three $\mathcal{F}_0$-solutions of in case $\kappa=0$,
      	  \item[(b)] precisely two $\mathcal{F}_0$-solutions of in case $0<|\kappa|\|B-A\|<M_0$,
      	  \item[(c)] precisely one $\mathcal{F}_0$-solution in case $|\kappa|\|B-A\|=M_0$,
      	  \item[(d)] no $\mathcal{F}_0$-solutions in case $|\kappa|\|B-A\|>M_0$
       \end{itemize}
      \item[] and for all $k\geq 1$ there are 
       \begin{itemize}
         \item[(e)] precisely four $\mathcal{F}_k$-solutions in case $|\kappa|\|B-A\|<2kC$,
         \item[(f)] precisely two $\mathcal{F}_k$-solutions in case $2kC\leq |\kappa|\|B-A\|<M_k$,
         \item[(g)] precisely one $\mathcal{F}_k$-solution in case $|\kappa|\|B-A\|=M_k$,
         \item[(h)] no $\mathcal{F}_k$-solutions in case $|\kappa|\|B-A\|>M_k$.
      \end{itemize}
      \item[(ii)] The minimal energy solutions of \eqref{Gl Navier BVP symm} are
      given as follows: 
      \begin{itemize}
        \item[(a)] In case $\kappa=0$ the Willmore minimizer is the straight line joining $A$ and $B$.
        \item[(b)] In case $2kC<|\kappa|\|B-A\|\leq M_k$ for some $k\in\N_0$ the Willmore minimizer is 
        of type $(-\sign(\kappa),\sign(\kappa),k)$.
        \item[(c)] In case $M_k<|\kappa|\|B-A\|\leq 2(k+1)C$ for some $k\in\N_0$ the Willmore minimizer is the
        unique solution of type $(\sign(\kappa),-\sign(\kappa),k)$.     
  	  \end{itemize}	    	  
  	  \item[(iii)] All $(\sigma_1,\sigma_2,j)$-type solutions with $\sigma_1=-\sigma_2$ are axially symmetric.
  	  The $(\sigma_1,\sigma_2,j)$-type solutions with $\sigma_1=\sigma_2$ are not axially symmetric and
  	  pointwise symmetric only in case $\kappa=0$.
    \end{itemize}
  \end{thm}
  
  We would like to point out that we can write down almost explicit formulas for the solutions obtained in
  Theorem \ref{Thm Navier kappa1=kappa2} using the formulas from our general result which we will provide in
  Theorem~\ref{Thm Navier BVP the general case} and the subsequent Remark~\ref{Bem allgemeines Resultat}~(b).
  Futhermore let us notice that the proof of Theorem \ref{Thm Navier kappa1=kappa2}~(ii)~(b) reveals which of
  the two $(-\sign(\kappa),\sign(\kappa),k)$-type solutions is the Willmore minimizer. We did not provide this
  piece of information in the above statement in order not to anticipate the required notation.
  At this stage let us only remark that in figure \ref{Fig 1} the Willmore minimizers from Theorem \ref{Thm
  Navier kappa1=kappa2}~(ii)~(b) lie on the left parts of the branches consisting of
  $(-\sign(\kappa),\sign(\kappa),0)$-type respectively $(-\sign(\kappa),\sign(\kappa),1)$-type solutions.

  \medskip
  
  \begin{bem} \label{Bem 1}~
    \begin{itemize}
      \item[(a)] The numerical value for $M_0$ is $1.34380$ which is precisely the threshold value
      found in Theorem~1 in~\cite{DecGru_Boundary_value_problems}. The existence result for graph-type
      solutions from \cite{DecGru_Boundary_value_problems} are therefore reproduced by Theorem \ref{Thm
      Navier kappa1=kappa2} (i)(b)-(d). Notice that a $(\sigma_1,\sigma_2,j)$-type solution with $j\geq 1$ can
      not be a graph with respect to any coordinate frame. On the other hand, in (i)(a) we found three
      $\mathcal{F}_0$-solutions for $\kappa=0$ which are the straight line and two nontrivial Willmore curves of type $(1,-1,0)$ respectively
      $(-1,1,0)$. The latter two solutions were not discovered by Deckelnick and Grunau because they are
      continuous but non-smooth graphs since the slopes (calculated with respect to the straight line
      joining the end points of the curve) at both end points are $+\infty$ or $-\infty$.
      Viewed as a curve, however, these solutions are smooth.
	  \item[(b)] From Theorem \ref{Thm Navier kappa1=kappa2} (i) one deduces that for all
	  $\kappa$ and all $A\neq B$ there is an infinite sequence of Willmore curves
	  solving \eqref{Gl Navier BVP symm} whose Willmore energies tends to $+\infty$. Here one uses that the Willmore energies of $\mathcal{F}_k-$solutions
	  tend to $+\infty$ as $k\to\infty$. We will see in Theorem~\ref{Thm Navier kappa1!=kappa2} that this result
	  remains true in the nonsymmetric case.
      \item[(c)] Using the formula for the length of a Willmore curve from our general result
      Theorem~\ref{Thm Navier BVP the general case} one can show that for any given $A,B\in\R^2,\kappa\in\R$
      satisfying $A\neq B,\kappa\neq 0$ the solution of \eqref{Gl Navier BVP symm} having minimal length $L$
      must be of type $(-\sign(\kappa),\sign(\kappa),j)$ for some $j\in\N_0$. More precisely, defining $L^*:=
      T\|B-A\|/(\sqrt{2}C)$ for the periodicity $T$ of $\cn$ (see section \ref{sec prelim}) one can show
      $L>L^*,L=L^*,L<L^*$ respectively if the solution curve is of type $(\sigma_1,\sigma_2,j)$ with
      $\sigma_1=-\sigma_2=\sign(\kappa), \sigma_1=\sigma_2,\sigma_1=-\sigma_2=-\sign(\kappa)$.
      Notice that in the case  $\|B-A\|=1$ studied by Linn{\'e}r we have $L^*\approx 2.18844$ which is
      precisely the value on page 460 in Linn{\'e}r's paper \cite{Lin_Explicit_elastic}.      
      Numerical plots indicate that in most cases the Willmore curves of minimal length have to be
      of type $(-\sign(\kappa),\sign(\kappa),j)$ where $j$ is smallest possible which, however, is not
      true in general as the following numerical example shows.
      In the special case $\|B-A\|=1,\kappa=9.885$ there is a $(-1,1,2)$-type solution 
      generated by $a\approx 7.48526$ having length $L_1\approx 2.08043$ and a $(-1,1,3)$-type
      solution generated by $a\approx 11.65140$ which has length $L_2\approx 2.08018$ and we observe
      $L_2<L_1$.
      \item[(d)] Having a look at figure \ref{Fig 1} we find that the set of Willmore minimizers described
      in  Theorem~\ref{Thm Navier kappa1=kappa2}~(ii) is not connected as a whole, but it is path-connected
      within suitable intervals of the parameter $\kappa\|B-A\|$.
    \end{itemize}
  \end{bem}
  
  In the nonsymmetric case $\kappa_1\neq \kappa_2$ it seems to be difficult to identify nicely
  behaved connected sets which constitute the whole solution set. However, we find existence and
  multiplicity results in the spirit of Remark \ref{Bem 1} (b). We can prove the following result.

  \begin{thm} \label{Thm Navier kappa1!=kappa2}
    Let $A,B\in\R^2, \kappa_1,\kappa_2\in\R$ satisfy $A\neq B,\kappa_1\neq \kappa_2$ and let
    $\sigma_1,\sigma_2\in\{-1,+1\}$. Then there is a smallest natural number $j_0$ such that  
    for all $j\geq j_0$ the Navier problem \eqref{Gl Navier BVP} has at least one
    $(\sigma_1,\sigma_2,j)$-type solution $\gamma_j$ such that $W(\gamma_{j_0}) < W(\gamma_{j_0+1}) < \ldots
    < W(\gamma_j) \to +\infty$ as $j\to\infty$. Moreover, we have $j_0\leq j_0^*$ where $j_0^*$
    is the smallest natural number satisfying 
    \begin{equation*} 
      \|B-A\|^2 
      \leq \begin{cases}       
      \frac{4C^2(j_0^*-\frac{1}{2})^2}{\max\{\kappa_1^2,\kappa_2^2\}}
      + \frac{4(\kappa_2-\kappa_1)^2}{\max\{\kappa_1^4,\kappa_2^4\}} &,\text{if }\sigma_1=\sigma_2\\
      \frac{4C^2(j_0^*-1)^2}{\max\{\kappa_1^2,\kappa_2^2\}}
      + \frac{4(\kappa_2-\kappa_1)^2}{\max\{\kappa_1^4,\kappa_2^4\}}&,\text{if } \sigma_1=-\sigma_2.
      \end{cases}    
    \end{equation*}
  \end{thm}   
  
  The proof of both theorems will be achieved by using an explicit expression for Willmore curves like the
  ones used in Linn{\'e}r's work \cite{Lin_Explicit_elastic} and Singer's lectures \cite{Sin_Lectures}. In
  Proposition \ref{Prop Navier} we show that the set of Willmore curves emanating from $A\in\R^2$ is
  generated by four real parameters so that solving the Navier problem is reduced to adjusting these
  parameters in such a way that the four equations given by the boundary conditions are satisfied. The proofs
  of Theorem \ref{Thm Navier kappa1=kappa2} and Theorem \ref{Thm Navier kappa1!=kappa2} will therefore be
  based on explicit computations which even allow to write down the complete solution theory for the general
  Navier problem \eqref{Gl Navier BVP}, see Theorem \ref{Thm Navier BVP the general case}.
  Basically the same strategy can be pursued to solve the general Dirichlet problem~\eqref{Gl Dirichlet
  BVP}. In view of the fact that a discussion of this boundary value problem does not reveal any new
  methodological aspect we decided to defer it to Appendix~B. Notice that even in the presumably simplest
  special cases $\theta_2-\theta_1\in\{-\pi/2,0,\pi/2\}$ it requires non-negligible computational effort to
  prove optimal existence and multiplicity results for \eqref{Gl Dirichlet BVP} like the ones in
  Theorem~\ref{Thm Navier kappa1=kappa2}. We want to stress that it would be desirable to prove the
  existence of infinitely many solutions of the Navier problem and the Dirichlet problem also in
  higher dimensions where Willmore curves $\gamma\in C^\infty([0,1],\R^n)$ with $n\geq 3$ are considered. It
  would be very interesting to compare such a solution theory to the corresponding theory for Willmore
  surfaces established by Sch\"atzle \cite{Sch_The_Willmore_boundary} where the existence of one solution of the
  corresponding Dirichlet problem is proved.  
  
  \medskip
  
  Let us finally outline how this paper is organized. In section \ref{sec prelim} we review some preliminary
  facts about Willmore curves and Jacobi's elliptic functions appearing in the definition of $\kappa_{a,b}$,
  see \eqref{Gl kappa_ab}. Next, in section \ref{sec Navier}, we provide all information concerning the Navier
  problem starting with a derivation of the complete solution theory for the general Navier problem \eqref{Gl
  Navier BVP} in subsection \ref{subsec general result}. In the subsections \ref{subsec Navier symmertric}
  and \ref{subsec Navier nonsymmertric} this result will be applied to the special cases $\kappa_1=\kappa_2$
  respectively $\kappa_1\neq \kappa_2$ in order to prove Theorem~\ref{Thm Navier kappa1=kappa2} and Theorem
  \ref{Thm Navier kappa1!=kappa2}. In Appendix~A we provide the proof of one auxiliary result (Proposition
  \ref{Prop Navier}) and Appendix~B is devoted to the solution theory for the general Dirichlet problem.

  \section{Preliminaries} \label{sec prelim}
   
  The Willmore functional is invariant under reparametrization so that the arc-length
  parametri\-zation $\hat\gamma$ of a given Willmore curve $\gamma$ is also a Willmore curve. It is known that
  the curvature function $\hat\kappa$ of $\hat\gamma$ satisfies the ordinary differential equation
  \begin{equation} \label{Gl Willmore equation}
    - \hat\kappa''(s) = \frac{1}{2}\hat\kappa(s)^3,
  \end{equation}  
  in particular it is not restrictive to consider smooth curves only. A derivation of \eqref{Gl Willmore
  equation} may be found in Singer's lectures \cite{Sin_Lectures}. Moreover, we
  will use the fact that every regular planar curve $\gamma\in C^\infty([0,1],\R^2)$ admits the
  representation
   \begin{align} \label{Gl ansatz gamma allgemein}
    \begin{aligned}
      \gamma\sim \hat\gamma,\qquad
      \hat\gamma(t) 
      =  A+ Q\int_0^t \vecII{\cos(\int_0^s \hat\kappa(r)\,dr)}{\sin(\int_0^s\hat\kappa(r)\,dr)} \,ds 
      \qquad (t\in [0,L]) 
    \end{aligned}
  \end{align}
  where $A$ is a point in $\R^2$ and $Q\in SO(2)$ is a rotation matrix. Here we used the symbol
  $\gamma\sim\hat\gamma$ in order to indicate that $\gamma,\hat\gamma$ can be transformed into each other by
  some diffeomorphic reparametrization. In Proposition \ref{Prop Navier} we will show that a regular curve
  $\gamma$ is a Willmore curve if and only if the function $\hat\kappa$ in
  \eqref{Gl ansatz gamma allgemein} is given by $\hat\kappa=\kappa_{a,b}$ for $\kappa_{a,b}$ as in \eqref{Gl
  kappa_ab} with ${a\geq 0},{L>0},{-T/2\leq b<T/2}$. In such a situation we will say that the
  parameters $a,b,L$ are admissible. Since this fact is fundamental for the proofs of our results we mention
  some basic properties of Jacobi's elliptic functions $\sn,\cn,\dn$. Each of these functions maps $\R$ into
  the interval $[-1,1]$ and is $T$-periodic where $T:= 4\int_0^1 (\tfrac{2}{1-t^4})^{1/2}\,dt \approx
  7.41630$. We will use the identities $\cn^4+2\sn^2\dn^2=1$ and $\cn'=-\sn\dn$ on $\R$ as well as
  $\arccos(\cn^2)'=\sqrt{2}\cn$ on $[0,T/2]$.
  Moreover we will exploit the fact that $\cn$ is $T/2$-antiperiodic, i.e. $\cn(x+T/2)=-\cn(x)$ for all
  $x\in\R$, and that $\cn$ is decreasing on $[0,T/2]$ with inverse function $\cn^{-1}$ satisfying
  $\cn^{-1}(1)=0$ and $(\cn^{-1})'(t)=-(\tfrac{2}{1-t^4})^{1/2}$ for $t\in (-1,1)$.

\section{The Navier problem} \label{sec Navier}

  \subsection{The general result} \label{subsec general result}
  
  The Navier boundary value problem \eqref{Gl Navier BVP} does not appear to be easily
  solvable since the  curvature function $\det(\gamma',\gamma'')|\gamma'|^{-3}$ of a given
  regular smooth curve $\gamma$ does in general not solve a reasonable differential equation.
  If, however, we parametrize $\gamma$ by arc-length we obtain a smooth regular curve
  $\hat\gamma:[0,L]\to\R^2$ whose  curvature function $\hat\kappa$ satisfies the differential equation
  $-2\hat\kappa''=\hat\kappa^3$ on $[0,L]$ which is explicitly solvable. Indeed, in Proposition~\ref{Prop
  Navier} we will show that the nontrivial solutions of this differential equation are given by the functions
  $\kappa_{a,b}$ from \eqref{Gl kappa_ab} so that the general form of a Willmore curve $\gamma$ is given by
  \begin{align} \label{Gl ansatz gamma}
    \begin{aligned}
      \gamma\sim\hat\gamma,\qquad
      \hat\gamma(t) 
      &=  A+ Q\int_0^t \vecII{\cos(\int_0^s
      \kappa_{a,b}(r)\,dr)}{\sin(\int_0^s\kappa_{a,b}(r)\,dr)} \,ds
      \qquad \text{for all }t\in [0,L]
    \end{aligned}
  \end{align} 
  where $Q\in SO(2)$ is a rotation matrix and $L>0$ denotes the length of $\gamma$.
  In such a situation we will say that the Willmore curve $\gamma$ is generated by $a,b,L$.  
  As a consequence, a Willmore curve $\gamma$ solves the Navier boundary value problem \eqref{Gl Navier BVP}
  once we determine the parameters $a,b,L,Q$ such that the equations
  \begin{align} \label{Gl Navier BC}
    \hat\gamma(L)-\hat\gamma(0)=B-A,\qquad \kappa_{a,b}(0)=\kappa_1,\qquad \kappa_{a,b}(L)=\kappa_2,
  \end{align}
  are satisfied. Summarizing the above statements and taking into account Definition \ref{Def sigma1sigma2j}
  we obtain the following preliminary result the proof of which we defer to Appendix A.
  
  \begin{prop} \label{Prop Navier}
    Let $\gamma\in C^\infty([0,1],\R^2)$ be a regular curve. Then the following is true:
    \begin{itemize}
      \item[(i)] The curve $\gamma$ solves the Navier problem \eqref{Gl Navier BVP} if and only if $\gamma$
      is given by \eqref{Gl ansatz gamma} for admissible parameters $a,b,L$ and $Q\in
      SO(2)$ satisfying \eqref{Gl Navier BC}.
      \item[(ii)] Let $\sigma_1,\sigma_2\in\{-1,+1\},j\in\N_0$ be given and assume $\gamma$ satisfies
      \eqref{Gl ansatz gamma}. Then $\gamma$ is of type $(\sigma_1,\sigma_2,j)$ if and only if there is
      some $\tilde b\in [-T/2,T/2)$ such that
      \begin{align}\label{Gl b+aL}
        \begin{aligned}
          b+aL = (j+\lceil(b-\tilde b)/T \rceil) T + \tilde b,\qquad   
          \sign(b)=\sigma_1, \qquad
          \sign(\tilde b)=\sigma_2.
        \end{aligned}
      \end{align}
    \end{itemize}
  \end{prop}
    
  In Theorem \ref{Thm Navier BVP the general case} we provide the complete solution
  theory for the boundary value problem \eqref{Gl Navier BVP} which, due to Proposition \ref{Prop Navier}, is
  reduced to solving the equivalent problem \eqref{Gl Navier BC}. In order to write down the
  precise conditions when $(\sigma_1,\sigma_2,j)$-type solutions of \eqref{Gl Navier BVP} exist we
  need the functions  
  \begin{align} \label{Gl def f}
    \begin{aligned} 
    f_{\sigma_1,\sigma_2,j}(\kappa_1,\kappa_2,a)
    &:= \Big(j+\Big\lceil
    \frac{\sigma_1\cn^{-1}(\frac{\kappa_1}{\sqrt{2}a})-\sigma_2\cn^{-1}(\frac{\kappa_2}{\sqrt{2}a})}{T}\Big\rceil\Big)\cdot
    C  \\
    &- \sigma_1 \begin{cases}
      \int_{\frac{\kappa_1}{\sqrt{2}a}}^{\frac{\kappa_2}{\sqrt{2}a}}\frac{t^2}{\sqrt{1-t^4}}\,dt &
      ,\text{if }\sigma_1=\sigma_2 \\
      \int_{\frac{\kappa_1}{\sqrt{2}a}}^1 \frac{t^2}{\sqrt{1-t^4}}\,dt+\int_{\frac{\kappa_2}{\sqrt{2}a}}^1
      \frac{t^2}{\sqrt{1-t^4}}\,dt &,\text{if }\sigma_1\neq \sigma_2 
    \end{cases}
    \end{aligned}
  \end{align}
  where $\kappa_1,\kappa_2\in\R$ and $a\geq a_0:=\max\{|\kappa_1|,|\kappa_2|\}/\sqrt{2}$.
  
  \begin{thm} \label{Thm Navier BVP the general case}
    Let $A,B\in\R^2,A\neq B$ and $\kappa_1,\kappa_2\in\R$ be given, let 
    $\sigma_1,\sigma_2\in \{-1,+1\},{j\in\N_0}$. Then all solutions of the Navier problem
    \eqref{Gl Navier BVP} which are regular curves of type $(\sigma_1,\sigma_2,j)$ are given by \eqref{Gl
    ansatz gamma} for $a,b,L$ and the uniquely determined\footnote{The unique
    solvability for the rotation matrix $Q$ will be ensured by arranging $a,b,L$ in such a way that the
    Euclidean norms on both sides of the equation are equal, see also Remark \ref{Bem allgemeines
    Resultat} (c).} rotation matrix $Q\in SO(2)$ satisfying
    \begin{align} \label{Gl equation for bQ}
      \begin{aligned}
        b &= \sigma_1 \cn^{-1}\Big(\frac{\kappa_1}{\sqrt{2}a}\Big),\qquad\qquad
        Q^T(B-A) = \int_0^L \vecII{\cos(\int_0^s \kappa_{a,b}(r))\,dr}{\sin(\int_0^s \kappa_{a,b}(r))\,dr},\\
        L &=  \frac{1}{a}\Big( \Big(j+ \Big\lceil
        \frac{\sigma_1\cn^{-1}(\frac{\kappa_1}{\sqrt{2}a})-\sigma_2\cn^{-1}(\frac{\kappa_2}{\sqrt{2}a})}{T}\Big\rceil\Big)T+
        \sigma_2\cn^{-1}\Big(\frac{\kappa_2}{\sqrt{2}a}\Big)-\sigma_1\cn^{-1}\Big(\frac{\kappa_1}{\sqrt{2}a}\Big)\Big)
      \end{aligned}
    \end{align}  
    where $a$ is any solution of
    \begin{align}\label{Gl eq for a}
      \begin{aligned}
      \|B-A\|^2 = \frac{2}{a^2} f_{\sigma_1,\sigma_2,j}(\kappa_1,\kappa_2,a)^2 +
      \frac{(\kappa_2-\kappa_1)^2}{a^4}\qquad (a>0,a\geq a_0) \\
      \text{provided}\quad      
      \kappa_1=\pm\sqrt{2}a\Rightarrow\sigma_1=\pm 1,\quad
      \kappa_2=\pm\sqrt{2}a\Rightarrow\sigma_2=\pm 1.
      \end{aligned}
    \end{align}
    The Willmore energy of such a curve $\gamma$ is given by
    \begin{align*}
       W(\gamma) 
       = \big( a^4 \|B-A\|^2- (\kappa_2-\kappa_1)^2\big)^{1/2}
    \end{align*}
  \end{thm}   
  \begin{proof}    
    By Proposition \ref{Prop Navier} we have to find all admissible parameters $a,b,L$ satisfying 
    \eqref{Gl b+aL} and a rotation matrix $Q\in SO(2)$ such that the equation \eqref{Gl Navier BC} holds.
    We show that all these parameters are given by \eqref{Gl equation for bQ} and \eqref{Gl eq for a} and it
    is a straightforward calculation (which we omit) to check that parameters given by these
    equations indeed give rise to a $(\sigma_1,\sigma_2,j)$-type solution of \eqref{Gl Navier BC}. So let
    us in the following assume that $a,b,L,Q$ solve \eqref{Gl Navier BC}. Notice that we have $a>0$ by
    assumption, i.e. $\gamma$ is not the straight line.
        
    \medskip
    
    First let us remark that the boundary conditions $\kappa_{a,b}(0)=\kappa_1,\kappa_{a,b}(L)=\kappa_2$ imply
    \begin{equation}\label{Gl metaThm 0}
      \cn(b)=\frac{\kappa_1}{\sqrt{2}a}, \qquad \cn(b+aL)=\frac{\kappa_2}{\sqrt{2}a}.
    \end{equation}
    so that $|\cn|\leq 1$ gives $\sqrt{2}a\geq \max\{|\kappa_1|,|\kappa_2|\}$ and thus $a\geq a_0$.

    \medskip
           
 	{\it 1st step: Simplification.}\; From \eqref{Gl Navier BC} we infer       
    \begin{align} \label{Gl metaThm 1}
      \|B-A\|^2
      &= \|\hat\gamma(L)-\hat\gamma(0)\|^2 \notag \\
      &= \Big\|\int_0^L Q\vecII{\cos(\int_0^s \kappa_{a,b}(r)\,dr)}{\sin(\int_0^s
      \kappa_{a,b}(r)\,dr)} \,ds\Big\|^2 \notag\\
      &= \Big\|\int_0^L \vecII{\cos(\int_0^s \kappa_{a,b}(r)\,dr)}{\sin(\int_0^s \kappa_{a,b}(r)\,dr)} \,ds\Big\|^2
      \notag \\
      &= \Big( \int_0^L \cos\Big(\int_0^s \kappa_{a,b}(r)\,dr\Big)\,ds \Big)^2
        + \Big( \int_0^L \sin\Big(\int_0^s \kappa_{a,b}(r)\,dr\Big)\,ds \Big)^2
    \end{align}
    Using the
	identities 
	\begin{align*}
	  \frac{d}{dr} \Big(2\arctan\Big(\frac{\sn(ar+b)}{\sqrt{2}\dn(ar+b)}\Big)\Big)
	  &=\sqrt{2}a\cn(ar+b)=\kappa_{a,b}(r) \qquad (r\in\R),  \\
	  \cos(2\arctan(x)-2\arctan(y))
	  &= \frac{2(1+xy)^2}{(1+x^2)(1+y^2)}-1
	  \;\;\qquad (x,y\in\R),\\
	  \sin(2\arctan(x)-2\arctan(y))
	  &= \frac{2(1+xy)(x-y)}{(1+x^2)(1+y^2)}
	  \qquad\qquad (x,y\in\R)
	\end{align*}
	we obtain after some rearrangements
    \begin{align} \label{Gl metaThm 2}
      \begin{aligned}
      \cos\Big(\int_0^s \kappa_{a,b}(r)\,dr\Big)
      &= \cn(b)^2 \cn(as+b)^2 + 2\sn(b)\dn(b)\sn(as+b)\dn(as+b), \\
      \sin\Big(\int_0^s \kappa_{a,b}(r)\,dr\Big)
      &= -\sqrt{2}\sn(b)\dn(b)\cn(as+b)^2 + \sqrt{2}\cn(b)^2\sn(as+b)\dn(as+b).
      \end{aligned}
    \end{align}
    In order to proceed with calculating $\|B-A\|^2$ we set
    \begin{equation}\label{Gl metaThm 3}
      \alpha:= \int_0^L \cn(as+b)^2\,ds, \qquad \beta:=\int_0^L \sqrt{2}\sn(as+b)\dn(as+b)\,ds.
    \end{equation}
    so that \eqref{Gl metaThm 1},\eqref{Gl metaThm 2} yield
    \begin{align}\label{Gl metaThm 6}
      \begin{aligned}
      \|B-A\|^2
      &= \Big( \cn(b)^2\alpha + \sqrt{2}\sn(b)\dn(b)\beta
      \Big)^2 + \Big( -\sqrt{2}\sn(b)\dn(b)\alpha + \cn(b)^2\beta \Big)^2 \\
      &= \alpha^2 (\cn(b)^4+2\sn(b)^2\cn(b)^2) + \beta^2 (\cn(b)^4+2\sn(b)^2\dn(b)^2) \\
      &= \alpha^2+\beta^2 \\
      &\stackrel{\eqref{Gl metaThm 3}}{=} \frac{1}{a^2} \Big( \int_b^{b+aL} \cn(t)^2\,dt \Big)^2 
         + \Big( \Big[-\frac{\sqrt{2}}{a}\cn(as+b)\Big]^L_0\Big)^2 \\
      &\stackrel{\eqref{Gl metaThm 0}}{=} \frac{1}{a^2} \Big( \int_b^{b+aL} \cn(t)^2\,dt \Big)^2 
         + \frac{(\kappa_2-\kappa_1)^2}{a^4}.
      \end{aligned}
    \end{align}
    In order to derive equation \eqref{Gl eq for a} it therefore remains to show 
    \begin{equation}\label{Gl metaThm 4}
      \frac{1}{\sqrt{2}} \int_b^{b+aL} \cn(t)^2\,dt=  f_{\sigma_1,\sigma_2,j}(\kappa_1,\kappa_2,a).
    \end{equation} 
    
    \medskip
    
    {\it 2nd step: Proving \eqref{Gl metaThm 4}.}\; For notational convenience we define $\tilde \kappa_1:=
    \kappa_1/(\sqrt{2}a),{\tilde \kappa_2:= \kappa_2/(\sqrt{2}a)}$ so that $a\geq a_0$ implies
    $\tilde\kappa_1,\tilde\kappa_2\in [-1,1]$. Then \eqref{Gl
    b+aL} and the boundary conditions \eqref{Gl metaThm 0} yield
    \begin{align}\label{Gl metaThm 7}
      \begin{aligned}
      b &= \sigma_1 \cn^{-1}(\tilde\kappa_1),\qquad\quad\tilde\kappa_1=\pm 1 \Rightarrow \sigma_1=\pm 1 \\
     \tilde b &= \sigma_2 \cn^{-1}(\tilde\kappa_2), \qquad\quad\tilde\kappa_2=\pm 1 \Rightarrow \sigma_2=\pm
     1.
     \end{aligned}
    \end{align}
    Using the shorthand notation $m:= j+\lceil (b-\tilde b)/T\rceil$ we obtain
    \begin{align*}
      &\; \frac{1}{\sqrt{2}}\int_b^{b+aL} \cn(t)^2\,dt \\ 
      &=  \frac{1}{\sqrt{2}}\int_{\sigma_1 \cn^{-1}(\tilde\kappa_1)}^{mT+\sigma_2
      \cn^{-1}(\tilde\kappa_2)} \cn(t)^2\,dt \\
      &=  \frac{m}{\sqrt{2}}\int_0^T \cn(t)^2\,dt +  \frac{1}{\sqrt{2}}\int_{\sigma_1
      \cn^{-1}(\tilde\kappa_1)}^{\sigma_2 \cn^{-1}(\tilde\kappa_2)} \cn(t)^2\,dt \\
      &= 2\sqrt{2}m\int_0^{T/4} \cn(t)^2\,dt +  \frac{\sigma_1}{\sqrt{2}}
        \begin{cases}
          \int_{\cn^{-1}(\tilde\kappa_1)}^{\cn^{-1}(\tilde\kappa_2)} \cn(t)^2\,dt &,\text{if }
          \sigma_1=\sigma_2\\
          -\int_0^{\cn^{-1}(\tilde\kappa_1)} \cn(t)^2\,dt 
           -\int_0^{\cn^{-1}(\tilde\kappa_2)} \cn(t)^2\,dt &,\text{if }\sigma_1\neq \sigma_2
        \end{cases} \\
     &=  2\sqrt{2}m\int_0^1  t^2\sqrt{\frac{2}{1-t^4}}\,dt - \frac{\sigma_1}{\sqrt{2}}
        \begin{cases}
           \int_{\tilde\kappa_1}^{\tilde\kappa_2} t^2\sqrt{\frac{2}{1-t^4}}\,dt  &,\text{if }
          \sigma_1=\sigma_2\\
          \int_{\tilde\kappa_1}^1 t^2\sqrt{\frac{2}{1-t^4}}\,dt 
           + \int_{\tilde\kappa_2}^1 t^2\sqrt{\frac{2}{1-t^4}}\,dt  &,\text{if }\sigma_1\neq \sigma_2
        \end{cases} \\
     &=  mC -\sigma_1
        \begin{cases}
          \int_{\tilde\kappa_1}^{\tilde\kappa_2} \frac{t^2}{\sqrt{1-t^4}}\,dt  &,\text{if }
          \sigma_1=\sigma_2,\\
          \int_{\tilde\kappa_1}^1\frac{t^2}{\sqrt{1-t^4}}\,dt 
           + \int_{\tilde\kappa_2}^1\frac{t^2}{\sqrt{1-t^4}}\,dt  &,\text{if }\sigma_1\neq\sigma_2
        \end{cases}   \\
      &= f_{\sigma_1,\sigma_2,j}(\kappa_1,\kappa_2,a).
    \end{align*}
    The formulas for $L$ and $W$ follow from \eqref{Gl b+aL} and \eqref{Gl metaThm 7}:
    \begin{align} \label{Gl length energy}
      \begin{aligned}
      L 
      &=  \frac{1}{a} \big( (j+\lceil (b-\tilde b)/T\rceil)T+\tilde b-b\big) \\
      &= \frac{1}{a}\Big( \Big(j+\Big\lceil
      \frac{\sigma_1\cn^{-1}(\tilde\kappa_1)-\sigma_2\cn^{-1}(\tilde\kappa_2)}{T}\Big\rceil\Big)T+\sigma_2
      \cn^{-1}(\tilde\kappa_2)-\sigma_1 \cn^{-1}(\tilde\kappa_1) \Big), \\
      W &= \frac{1}{2}\int_\gamma \kappa_\gamma^2 
      \; = \; \frac{1}{2}\int_0^L \kappa_{a,b}(s)^2\,ds
      \; = \; a^2 \int_0^L \cn(as+b)^2\,ds\\
        &\stackrel{\eqref{Gl metaThm 3}}{=} a^2\cdot \alpha 
      \; \stackrel{\eqref{Gl metaThm 6}}{=} \;a^2 \cdot (\|B-A\|^2-\beta^2)^{1/2} \\
        &=  \big( a^4 \|B-A\|^2-(\kappa_2-\kappa_1)^2\big)^{1/2}.
      \end{aligned}
    \end{align}
  \end{proof}

   \begin{bem} \label{Bem allgemeines Resultat}~ 
     \begin{itemize}       
       \item[(a)] Reviewing the proof of Theorem \ref{Thm Navier BVP the general case} one finds that the
       boundary value problem \eqref{Gl Navier BVP} does not have nontrivial solutions in case
       $A=B$, i.e. there are no closed Willmore curves. This result has already been obtained by Linn\'{e}r,
       cf. Proposition 3 in \cite{Lin_Explicit_elastic}. In fact, \eqref{Gl metaThm 6} would imply $aL=0$ so
       that the solution must be the ''trivial curve'' consisting of precisely one point.        
       \item[(b)] For computational purposes it might be interesting to know that the integrals involved in
       the definition of $Q$ need not be calculated numerically. With some computational effort one finds that
       $Q$ is the unique rotation matrix satisfying $Qw=B-A$ where the vector $w=(w_1,w_2)\in\R^2$ is given
       by
       \begin{align*}
    	  w_1
          &= \frac{\kappa_1^2(a^4 \|B-A\|^2-(\kappa_2-\kappa_1)^4)^{1/2} + 
          \sigma_1(\kappa_1-\kappa_2)(4a^4-\kappa_1^4)^{1/2} }{2a^4},   \\
    	  w_2 &= \frac{-\sigma_1 ((4a^4-\kappa_1^4)(a^4\|B-A\|^2-(\kappa_2-\kappa_1)^4))^{1/2} + 
    	  \kappa_1^2(\kappa_1-\kappa_2)}{2a^4}.
  	   \end{align*}
  	   Hence, $Q$ can be expressed in terms of $B-A,w_1,w_2$ only.
     \end{itemize}    
   \end{bem}

  \subsection{Proof of Theorem \ref{Thm Navier kappa1=kappa2}} \label{subsec Navier symmertric} ~ 
  
  \medskip
  
  \subsubsection{Proof of Theorem \ref{Thm Navier kappa1=kappa2} (i) - Existence results} ~ 
  
  \medskip
  
  In order to prove Theorem \ref{Thm Navier kappa1=kappa2} (i) we apply Theorem \ref{Thm Navier BVP the
  general case} to the special case $\kappa_1=\kappa_2=\kappa$. As before we set
  $\sqrt{2}a_0:= \max\{|\kappa_1|,|\kappa_2|\}=|\kappa|$. Theorem \ref{Thm Navier BVP the general case} tells
  us that $a>0,a\geq a_0$ generates a solution of type $(\sigma_1,\sigma_2,j)$ if
  and only if
  \begin{align*} 
    \begin{aligned}
     \|B-A\| 
     &= \frac{\sqrt{2}}{a} f_{\sigma_1,\sigma_2,j}(\kappa,\kappa,a) \\
     &= \frac{\sqrt 2}{a} \Big( 
        \Big(j+\Big\lceil
      \frac{(\sigma_1-\sigma_2)\cn^{-1}(\frac{\kappa}{\sqrt{2}a})}{T}\Big\rceil \Big)C     
       - \sigma_1 \begin{cases}
      0 &, \text{if }\sigma_1=\sigma_2, \\
      2\int_{\frac{\kappa}{\sqrt{2}a}}^1 \frac{t^2}{\sqrt{1-t^4}}\,dt  &, \text{if }\sigma_1\neq \sigma_2 
    \end{cases}  \Big) \\
    &= \frac{\sqrt 2}{a} \cdot
      \begin{cases}
        jC &, \text{if }\sigma_1=\sigma_2 \text{ or }(\sigma_1,\sigma_2,\frac{\kappa}{\sqrt{2}a})=(-1,1,-1),
        \\
        jC+2\int_{\frac{\sigma_2\kappa}{\sqrt{2}a}}^1 \frac{t^2}{\sqrt{1-t^4}}\,dt  &, \text{if
        }\sigma_1\neq \sigma_2 \text{ and }(\sigma_1,\sigma_2,\frac{\kappa}{\sqrt{2}a})\neq (-1,1,-1)
        \end{cases} 
     \end{aligned}
  \end{align*}  
  provided $|\kappa|=\sqrt{2}a$ implies $\sigma_1=\sigma_2=\sign(\kappa)$. Hence, the
  above equation  is equivalent to \\~\\~\\
  \begin{subequations} \label{Gl main equation symmetric} 
    \renewcommand\theequation*{(\theparentequation)$_j$}
    \vspace{-3\baselineskip}     
      \begin{align*} 
        &\|B-A\|
      =  \frac{\sqrt 2}{a} \cdot
      \begin{cases}
        jC &, \text{if }\sigma_1=\sigma_2,   \\
        jC+2\int_{\frac{\sigma_2\kappa}{\sqrt{2}a}}^1 \frac{t^2}{\sqrt{1-t^4}}\,dt  &, \text{if
        }\sigma_1\neq \sigma_2
        \end{cases} \\
     &\text{provided}\quad |\kappa|=\sqrt{2}a \Rightarrow \sigma_1=\sigma_2=\sign(\kappa)
      \end{align*}
  \end{subequations}
  
  so that it remains to determine all solutions $a>0,a\geq a_0$ of \eqref{Gl main equation symmetric}$_j$.
  First let us note that there are no $(\sigma_1,\sigma_2,0)$-type
  solutions with $\sigma_1=\sigma_2$, see \eqref{Gl main equation symmetric}$_0$. Hence, by definition of the
  sets $\mathcal{F}_0,\mathcal{F}_1,\ldots,$ 
  all solutions of the symmetric Navier problem \eqref{Gl Navier BVP symm} are contained in
  $\bigcup_{k=0}^\infty\mathcal{F}_k$ which shows $\mathcal{F}=\bigcup_{k=0}^\infty\mathcal{F}_k$. In order
  to prove the existence and nonexistence results stated in Theorem \ref{Thm Navier kappa1=kappa2} it
  therefore remains to find all solutions belonging to $\mathcal{F}_k= \mathcal{F}_k^+ \cup \mathcal{F}_k^-$
  for any given $k\in\N_0$.
  In the following we will without loss of generality restrict ourselves to the case $\kappa\geq 0$. This is
  justified since every $(\sigma_1,\sigma_2,j)$-type solution of \eqref{Gl Navier BVP symm} generated by
  $a,b,L$ gives rise to a $(-\sigma_1,-\sigma_2,j)$-type solution of \eqref{Gl Navier BVP symm} with $\kappa$
  replaced by $-\kappa$ which is generated by $a,b-\sigma_1T/2,L$. Let us recall that figure
  \ref{Fig 1} might help to understand and to visualize the proof of Theorem \ref{Thm Navier kappa1=kappa2}
  (i).

  \medskip

  {\it Proof of Theorem \ref{Thm Navier kappa1=kappa2} (i)(a)-(d):}\;
  First we determine all solutions belonging to $\mathcal{F}_0^+$. All solutions from $\mathcal{F}_0^+$ with
  $\kappa=0$ are the straight line from $A$ to $B$ and the solution of type
  $(-1,1,0)$ which is generated by $a= C(\sqrt{2}\|B-A\|)^{-1}$, see \eqref{Gl
  main equation symmetric}. The solutions from $\mathcal{F}_0^+$ with $\kappa>0$ are of type
  $(-1,1,0)$ and therefore generated by $a>0,a\geq a_0$ such that
  \begin{equation} \label{Gl F0 I}
    \kappa\|B-A\| = 4z\int_z^1 \frac{t^2}{\sqrt{1-t^4}}\,dt
    \qquad\text{where } z = \frac{\kappa}{\sqrt{2}a} \in (0,1),
  \end{equation}
  see \eqref{Gl main equation symmetric}. Notice that $z>0$ is due to the assumption $\kappa>0$ whereas $z\leq
  1$ follows from $a\geq a_0$ and $z\neq 1$ follows from the last line in \eqref{Gl main equation symmetric}. The
  function defined by the right hand side is positive and strictly concave on $(0,1)$ with global maximum $M_0$. Hence, we
  find two, one respectively no solution in $\mathcal{F}_0^+$ according to $\kappa\|B-A\|$ being smaller
  than, equal to or larger than $M_0$. This proves the assertions (i)(a)-(d) in Theorem \ref{Thm Navier kappa1=kappa2}.

  \medskip

  {\it Proof of Theorem \ref{Thm Navier kappa1=kappa2} (i)(e)-(h):}\; Now we investigate the solutions
  belonging to $\mathcal{F}_k^+$ for $k\geq 1$. From \eqref{Gl main equation symmetric} we infer that solutions of type
  $(\sigma_1,\sigma_2,k)$ with $\sigma_1=\sigma_2$ exist if and only if
  \begin{align*}
    \sigma_1=\sigma_2=-1,\; 0\leq\kappa\|B-A\|<2kC
    \quad\text{ or }\quad
    \sigma_1=\sigma_2=1,\; 0\leq \kappa\|B-A\|\leq 2kC.
  \end{align*}
  In both cases the solution is unique and given by $a=\sqrt{2}kC\|B-A\|^{-1}$.
  (Notice that the special case $\kappa\|B-A\|=2kC$ corresponds to $\kappa=\sqrt{2}a$ so that the last
  line in \eqref{Gl main equation symmetric} is responsable for the slight difference between the cases
  $\sigma_1=\sigma_2=-1$ and $\sigma_1=\sigma_2=1$.)

  \medskip

  Next we study the solutions belonging to $\mathcal{F}_k^+$ which are of type
  $(1,-1,k-1)$. According to \eqref{Gl main equation symmetric} these solutions are
  generated by all $a>0,a\geq a_0$ which satisfy
  \begin{equation} \label{Gl Fk I}
    \kappa\|B-A\| = 2z\Big((k-1)C+2\int_{-z}^1 \frac{t^2}{\sqrt{1-t^4}}\,dt\Big)
    \qquad\text{where }z= \frac{\kappa}{\sqrt{2}a} \in [0,1).
  \end{equation}
  Here, $z\in [0,1)$ follows from $a\geq a_0$ and the last line in \eqref{Gl
  main equation symmetric}. Since the function defined by the right hand side increases from 0 to
  $2kC$ we have precisely one solution of type $(1,-1,k-1)$ for
  $0\leq \kappa \|B-A\|<2kC$ and no such solution otherwise.

  \medskip

  Finally we study the solutions in $\mathcal{F}_k^+$ which are of type $(-1,1,k)$.
  For this kind of solutions equation \eqref{Gl main equation symmetric} reads
  \begin{equation} \label{Gl Fk II}
    \kappa\|B-A\| = 2z\Big(kC+2\int_z^1 \frac{t^2}{\sqrt{1-t^4}}\,dt\Big)
    \qquad\text{where }z= \frac{\kappa}{\sqrt{2}a} \in [0,1).
  \end{equation}
  The function given by the right hand side is positive and strictly concave on $[0,1)$ with global
  maximum $M_k$ and which attains the values 0 at $z=0$ and $2kC$ at $z=1$. Hence we find precisely one
   $(-1,1,k)$-type solution for $0\leq \kappa\|B-A\|\leq 2kC$, two $(-1,1,k)$-type solutions for $2kC<
  \kappa\|B-A\|<M_k$, one $(-1,1,k)$-type solution for $\kappa\|B-A\|=M_k$ and no such solution otherwise.

\medskip

  \subsubsection{Proof of Theorem \ref{Thm Navier kappa1=kappa2} (ii) - Willmore curves of minimal Willmore
  energy} ~
  
  \medskip
  
  In case $\kappa=0$ the straight line joining $A$ and $B$ is the unique minimal energy solution of~\eqref{Gl
  Navier BVP symm} so that it remains to analyse the case $\kappa\neq 0$. According to Theorem~\ref{Thm
  Navier kappa1=kappa2} the energy of a nontrivial $(\sigma_1,\sigma_2,j)$-type solution of \eqref{Gl Navier
  BVP symm} is given by $W = a^2\|B-A\|^2$ where $a$ solves~\eqref{Gl main equation symmetric}$_j$. Hence,
  finding the minimal energy solution of \eqref{Gl Navier BVP symm} is equivalent to finding
  $\sigma_1,\sigma_2\in\{-1,+1\},j\in\N_0$ and a solution $a$ of \eqref{Gl main equation
  symmetric}$_j$ such that $a$ is smallest possible. To this end we will dinstinguish the cases
  $$
    \text{(a)}\; 2kC<|\kappa|\|B-A\|\leq M_k\quad (k\in\N_0)
    \qquad
    \text{(b)}\; M_k<|\kappa|\|B-A\|\leq 2(k+1)C\quad (k\in\N_0).
  $$
  First let us note that in the above situation we only have to consider $(\sigma_1,\sigma_2,j)$-type
  solutions where $j$ is smallest possible.

  \medskip

  In case (a) there are two $(-\sign(\kappa),\sign(\kappa),k)$-type solutions (belonging to $\mathcal{F}_k$)
  and one $(\sign(\kappa),-\sign(\kappa),k)$-type solution (belonging to $\mathcal{F}_{k+1}$), see figure
  \ref{Fig 1}. All other $(\sigma_1,\sigma_2,j)$-type solutions satisfy $j\geq k+1$ and need not be taken into
  account by the observation we made above. According to the formulas \eqref{Gl F0 I},\eqref{Gl Fk
  I},\eqref{Gl Fk II} the smallest possible $a$ occurs when $\sigma_1=-\sign(\kappa),\sigma_2=\sign(\kappa)$
  and $z:=|\kappa|/(\sqrt{2}a)\in (0,1)$ is the larger solution of \eqref{Gl F0 I} in case $k=0$ respectively
  \eqref{Gl Fk II} in case $k\geq 1$. 

  \medskip

  In case (b) there is one $(\sign(\kappa),-\sign(\kappa),k)$-type solution whereas all other solutions
  are of type $(\sigma_1,\sigma_2,j)$ with $j\geq k+1$. Hence, this Willmore curve has minimal
  Willmore energy among all solutions of \eqref{Gl Navier BVP symm}.
  
  \medskip
  
  \subsubsection{Proof of Theorem \ref{Thm Navier kappa1=kappa2} (iii) - Symmetry results} ~
  
  \medskip
  
  In the following let $R(\tau)\in SO(2)$ denote the orthogonal matrix which realizes a rotation through
  the rotation angle $\tau\in\R$, i.e.
  $$
    R(\tau) := \matII{\cos(\tau)}{-\sin(\tau)}{\sin(\tau)}{\cos(\tau)}.
  $$
  With the following Proposition at hand the proof of our symmetry result for the solutions
  $\gamma:[0,1]\to\R^2$ from Theorem \ref{Thm Navier kappa1=kappa2} is
  reduced to proving symmetry results of the  curvature function
  $\hat\kappa=\det(\hat\gamma',\hat\gamma''):[0,L]\to\R$ of its arc-length parametrization
  ${\hat\gamma:[0,L]\to\R^2}$. Let us recall that $\hat\kappa$ is said to be symmetric in case
  $\hat\kappa(L-s)=\hat\kappa(s)$ for all $s\in [0,L]$ and it is called pointwise symmetric (with respect to
  its midpoint) in case ${\hat\kappa(L-s)=-\hat\kappa(s)}$ for all $s\in [0,L]$.

  \begin{prop} \label{Prop symmetry}
    Let $\gamma:[0,1]\to\R^2$ be a twice differentiable regular curve and let
    ${\hat\gamma:[0,L]\to\R^2}$ denote its parametrization by arc-length with
     curvature function $\hat\kappa$. Then $\gamma$ is
    \begin{itemize}
      \item[(i)] axially symmetric if and only if $\hat\kappa$ is symmetric,
      \item[(ii)] pointwise symmetric if and only if $\hat\kappa$ is pointwise symmetric.
    \end{itemize}
  \end{prop}
  \begin{proof}
    Let us start with proving (i). We first assume that $\hat\kappa$ is symmetric, i.e.
    $\hat\kappa(s)=\hat\kappa(L-s)$ for all $s\in\R$. We have to prove that there is a $P\in O(2)$ with
    $\det(P)=-1$ such that
    $\hat\gamma(L-s)+P\hat\gamma(s)$ is constant. Indeed, the symmetry assumption on $\hat\kappa$ implies
    $$
      \theta+\int_0^{L-s}  \hat\kappa(r)\,dr + \int_0^s \hat\kappa(r)\,dr = \theta + \int_0^L
      \hat\kappa(r)\,dr =:\theta_L \qquad\text{for all }s\in [0,L]
    $$
    and we obtain from \eqref{Gl ansatz gamma allgemein}
    \begin{align*}
      \frac{d}{ds}\big( \hat\gamma(L-s)\big)
      &= -\hat\gamma'(L-s)\\
      &= \vecII{-\cos(\theta+\int_0^{L-s}\hat\kappa(r)\,dr)}{-\sin(\theta+\int_0^{L-s}\hat\kappa(r)\,dr)} \\
      &= \vecII{-\cos(\theta_L -\int_0^s\hat\kappa(r)\,dr)}{-\sin(\theta_L -\int_0^s\hat\kappa(r)\,dr)}  \\
      &= - R(\theta_L) \vecII{\cos(-\int_0^s\hat\kappa(r)\,dr)}{\sin(-\int_0^s\hat\kappa(r)\,dr)}  \\
      &= - R(\theta_L)
      \matII{1}{0}{0}{-1}\vecII{\cos(\int_0^s\hat\kappa(r)\,dr)}{\sin(\int_0^s\hat\kappa(r)\,dr)} \\
      &= - R(\theta_L)
      \matII{1}{0}{0}{-1}R(-\theta)
      \vecII{\cos(\theta+\int_0^s\hat\kappa(r)\,dr)}{\sin(\theta+\int_0^s\hat\kappa(r)\,dr)} \\
      &= -P\hat\gamma'(s)
      \qquad\text{where }P :=  R(\theta_L) \matII{1}{0}{0}{-1}R(-\theta).
    \end{align*}
    Since $P\in O(2)$ satisfies $\det(P)=-1$ we obtain that $\gamma$ is axially symmetric. Vice
    versa, the condition $\hat\kappa(s)=\hat\kappa(L-s)$ for all $s\in [0,L]$ is also necessary for
    $\gamma$ to be axially symmetric since $\hat\gamma(L-s)+P\hat\gamma(s)=const$ implies
    $\hat\gamma'(L-s)=P\hat\gamma'(s),\hat\gamma''(L-s)=-P\hat\gamma''(s)$ and thus
    \begin{align*}
      \hat\kappa(L-s)
      &= \det(\hat\gamma'(L-s),\hat\gamma''(L-s)) \\
      &= \det(P\hat\gamma'(s),-P\hat\gamma''(s)) \\
      &= -\det(P)\det(\hat\gamma'(s),\hat\gamma''(s)) \\
      &= \hat\kappa(s) \qquad\text{for all }s\in [0,L].
    \end{align*}

    \medskip

    The proof of (ii) is similar. Assuming $\hat\kappa(s)=-\hat\kappa(L-s)$ for all $s\in [0,L]$ one deduces
    $\theta=\theta_L$ and thus
    $$
      \theta+\int_0^{L-s} \hat\kappa(r)\,dr - \int_0^s \hat\kappa(r)\,dr = \theta_L = \theta
      \qquad\text{for all }s\in [0,L]
    $$
    which, as above, leads to $\hat\gamma(L-s)+\hat\gamma(s)=const$.
  \end{proof}

  \medskip

  In our proof of Theorem \ref{Thm Navier kappa1=kappa2} (iii) we will apply Proposition \ref{Prop symmetry}
  to Willmore curves which, according to Proposition \ref{Prop Navier}, are given by \eqref{Gl ansatz gamma}.
  In this context the above result reads as follows.

  \begin{cor} \label{Cor symmetry}
    Let $\gamma:[0,1]\to\R^2$ be a Willmore curve given by \eqref{Gl ansatz gamma}. Then $\gamma$ is
    \begin{itemize}
      \item[(i)] axially symmetric if and only if $aL+2b$ is an even multiple of $T/2$,
      \item[(ii)] pointwise symmetric if and only if $aL+2b$ is an odd multiple of $T/2$.
    \end{itemize}
  \end{cor}
  \begin{proof}
    We apply Proposition \ref{Prop symmetry} to the special case $\hat\kappa=\kappa_{a,b}:[0,L]\to\R$
    for admissible parameters $a,b,L$. Hence, both (i) and (ii) follow from
    \begin{align*}
      \kappa_{a,b}(L-s)
      &= \sqrt{2}a\cn(a(L-s)+b) \\
      &= \sqrt{2}a\cn(-a(L-s)-b) \\
      &= \sqrt{2}a\cn(as+b-(aL+2b))
    \end{align*}
    and $\cn(z+kt)=(-1)^k\cn(z)$ for all $z\in\R,k\in\Z$ if and only if $t=T/2$.
  \end{proof}

  \medskip

  {\it Proof of Theorem \ref{Thm Navier kappa1=kappa2} (iii):}\; Let $\kappa\in\R$,
  $\sigma_1,\sigma_2\in\{-1,+1\}$ and $j\in\N_0$. The proof of Theorem \ref{Thm Navier kappa1=kappa2}
  tells us that every $(\sigma_1,\sigma_2,j)$-type solution of the symmetric Navier Problem \eqref{Gl Navier
  BVP symm} is generated by admissible parameters $a,b,L$ which satisfy equation \eqref{Gl b+aL}, in
  particular
  $$
    aL+2b
    = mT+(\sigma_2+\sigma_1)\cn^{-1}\big(\frac{\kappa}{\sqrt{2}a}\big)
    \qquad\text{for some }m\in\N_0.
  $$
  In case $\sigma_1=-\sigma_2$ this implies that $aL+2b$ is an even multiple of $T/2$ which, by
  Corollary~\ref{Cor symmetry}~(i), proves that the solution is axially symmetric. In case
  $\sigma_1=\sigma_2$ we find that the solution is axially symmetric if and only if
  $2\sigma_1\cn^{-1}\big(\tfrac{\kappa}{\sqrt{2}a}\big)$ is a multiple of $T$. This is equivalent to
  $\cn^{-1}\big(\tfrac{\kappa}{\sqrt{2}a}\big)\in \{0,T/2\}$ and thus to $|\kappa|=\sqrt{2}a$. Due to the
  restriction $z<1$ in \eqref{Gl F0 I},\eqref{Gl Fk I},\eqref{Gl Fk II} this equality cannot hold true so
  that these solutions are not axially symmetric. Finally we find that $\gamma$ is pointwise
  symmetric if and only if $\cn^{-1}\big(\tfrac{\kappa}{\sqrt{2}a}\big) = T/4$, i.e. if and only if
  $\kappa=0$.~\qed

  \subsection{Proof of Theorem \ref{Thm Navier kappa1!=kappa2}} \label{subsec Navier nonsymmertric} ~
  
  \medskip
  
  {\it 1st step:}\, We first show that a minimal index $j_0$ having the properties claimed in Theorem \ref{Thm
  Navier kappa1!=kappa2} exists. To this end let $\sigma_1,\sigma_2\in\{-1,+1\}$ and $k,l\in\N_0$ with $k<l$
  be arbitrary. Our aim is to show that if a $(\sigma_1,\sigma_2,k)$-type solution generated by a solution
  $a_k\in [a_0,\infty)$ of \eqref{Gl eq for a}$_k$ exists then there is a $(\sigma_1,\sigma_2,l)$-solution
  generated by some $a_l\in (a_k,\infty)$ and every such solution satisfies $a_l>a_k$ (i.e. $a_0\leq a_l\leq
  a_k$ cannot hold). Once this is shown it follows that there is a minimal index
  $j_0=j_0(\sigma_1,\sigma_2)$ such that for all $j\geq j_0$ a $(\sigma_1,\sigma_2,j)$-type solution
  $\gamma_j$ of \eqref{Gl Navier BVP} exists with
   $W(\gamma_{j_0})<W_{\gamma_{j_0+1}}<\ldots<W(\gamma_j)\to\infty$ as $j\to\infty$. The latter statement
  follows from $a_j\to \infty$ as $j\to\infty$. So it remains to prove the statement from above.
  
  \medskip
  
    For $l>k$ we have
    \begin{align*}
      \frac{2f_{\sigma_1,\sigma_2,l}(a_k,\kappa_1,\kappa_2)}{a_k^2} &+ \frac{(\kappa_2-\kappa_1)^2}{a_k^4}
      > \frac{2f_{\sigma_1,\sigma_2,k}(a_k,\kappa_1,\kappa_2)}{a_k^2} + \frac{(\kappa_2-\kappa_1)^2}{a_k^4}
      \stackrel{\eqref{Gl eq for a}_k}{=} \|B-A\|^2, \\
      &\text{and }\quad \lim_{a\to \infty} \Big( \frac{2f_{\sigma_1,\sigma_2,l}(a,\kappa_1,\kappa_2)}{a^2} +
      \frac{(\kappa_2-\kappa_1)^2}{a^4}\Big) = 0
    \end{align*}
    so that the intermediate value theorem provides at least one solution $a_l\in (a_k,\infty)$ of
    \eqref{Gl eq for a}$_l$. In addition, every such solution has to lie in $(a_k,\infty)$ because $a_0\leq
    a\leq a_k$ implies
    \begin{align*}
      \frac{2f_{\sigma_1,\sigma_2,l}(a,\kappa_1,\kappa_2)}{a^2} + \frac{(\kappa_2-\kappa_1)^2}{a^4}
      &\geq \frac{2f_{\sigma_1,\sigma_2,k}(a_k,\kappa_1,\kappa_2)}{a^2} + \frac{(\kappa_2-\kappa_1)^2}{a^4} \\
      &\geq \frac{2f_{\sigma_1,\sigma_2,k}(a_k,\kappa_1,\kappa_2)}{a_k^2} +
       \frac{(\kappa_2-\kappa_1)^2}{a_k^4}   \\
      &\stackrel{\eqref{Gl eq for a}_k}{=} \|B-A\|^2,
    \end{align*}
    where equality cannot hold simultaneously in both inequalities. Hence, equation \eqref{Gl eq for a}$_l$
    does not have a solution in the interval $[a_0,a_k]$ which finishes the first step.
  
  \medskip
  
  {\it 2nd step:}\, For given $\sigma_1,\sigma_2\in\{-1,+1\}$ let us show $j_0\leq j_0^*$ or equivalently that
  for all $j\geq j_0^*$ there is at least one $(\sigma_1,\sigma_2,j)$-type solution of the Navier problem
  \eqref{Gl Navier BVP}. So let $j\geq j_0^*$. Then \eqref{Gl def f} implies
  \begin{align*}
    &f_{\sigma_1,\sigma_2,j}(\kappa_1,\kappa_2,a_0)
    \geq (j-\frac{1}{2}) C
     \geq (j_0^*-\frac{1}{2}) C
     &&\text{if } \sigma_1=\sigma_2, \\
    &f_{\sigma_1,\sigma_2,j}(\kappa_1,\kappa_2,a_0)
    \geq (j-1)C
     \geq (j_0^*-1)C
     &&\text{if } \sigma_1=-\sigma_2
  \end{align*}
  so that the choice for $j_0^*$ from Theorem \ref{Thm Navier kappa1!=kappa2} and
  $\sqrt{2}a_0=\max\{|\kappa_1|,|\kappa_2|\}$ imply
  \begin{align*}
    \frac{2f_{\sigma_1,\sigma_2,j}(\kappa_1,\kappa_2,a_0)^2}{a_0^2}  &+ \frac{(\kappa_2-\kappa_1)^2}{a_0^4}
    =\frac{4 f_{\sigma_1,\sigma_2,j}(\kappa_1,\kappa_2,a_0)^2}{\max\{\kappa_1^2,\kappa_2^2\}} +
    \frac{4(\kappa_2-\kappa_1)^2}{\max\{\kappa_1^4,\kappa_2^4\}}
    \geq \|B-A\|^2, \\
    &\text{and }\quad \lim_{a\to \infty} \Big( \frac{2f_{\sigma_1,\sigma_2,j}(\kappa_1,\kappa_2,a)}{a^2} +
    \frac{(\kappa_2-\kappa_1)^2}{a^4}\Big) = 0.
  \end{align*}
  Hence, the intermediate value theorem provides at least one solution $a\in [a_0,\infty)$ of equation
  \eqref{Gl eq for a}$_j$ which, by Theorem \ref{Thm Navier BVP the general case}, implies that there is at
  least one $(\sigma_1,\sigma_2,j)$-type solution of the Navier problem \eqref{Gl Navier BVP}. 
  ~\qed
%

\section*{Appendix A -- Proof of Proposition \ref{Prop Navier}}
   
   \medskip 
    
  {\it Proof of part (i):}\;
   In view of the formula \eqref{Gl ansatz gamma allgemein} for regular curves it remains to prove that 
   for every $\kappa_0,\kappa_0'\in\R$ with $(\kappa_0,\kappa_0')\neq (0,0)$ the unique solution of the
   initial value problem 
   \begin{equation} \label{Gl Prop1 I}
     - \hat\kappa''(s) = \frac{1}{2}\hat\kappa(s)^3,\qquad \hat\kappa(0)=\kappa_0,\;\hat\kappa'(0)=\kappa_0'
   \end{equation}
   is given by $\kappa_{a,b}(s) = \sqrt{2}a \cn(as+b)$ where $a>0, b\in [-T/2,T/2)$ are chosen according to
   \begin{equation} \label{Gl Prop1 II}
     a= \big( \frac{1}{4}\kappa_0^4 + {\kappa_0'}^2\big)^{1/4},\qquad
     b= -\sign(\kappa_0') \cn^{-1}(\frac{\kappa_0}{\sqrt{2}a}).
   \end{equation}
   Indeed, \eqref{Gl Prop1 II} implies $\kappa_{a,b}(0)=
   \kappa_0$ and the equations $\cn'=-\sn\dn, \cn^4+2\sn^2\dn^2=1$ give
   \begin{align*}
     \kappa_{a,b}'(0)
     &=  \sqrt{2} a^2 \cn'(b)
     \;=\; -a^2 \cdot \sqrt{2}\sn(b)\dn(b)
     \:=\; -a^2 \sign(b)\sqrt{1-\cn^4(b)} \\
     &= a^2 \sign(\kappa_0')\sqrt{1-\frac{\kappa_0^4}{4a^4}}
     \;=\; \sign(\kappa_0')|\kappa_0'| \\
     &= \kappa_0'.
   \end{align*}
   In view of unique solvability of the initial value problem \eqref{Gl Prop1 I} this proves the claim.
   
   \medskip
   
   {\it Proof of part (ii):}\; This follows from setting $\tilde b:= b+aL -mT$ for $m\in\N_0$  such that
   $\tilde b\in [-T/2,T/2)$. \qed
   
\section{Appendix B -- The Dirichlet problem}

  In this section we want to provide the solution theory for the Dirichlet problem~\eqref{Gl Dirichlet
  BVP}. As in our analysis of the Navier problem we first reduce the boundary value problem \eqref{Gl
  Dirichlet BVP} for the regular curve $\gamma:[0,1]\to\R^2$ to the corresponding boundary value
  problem for its parametrization by arc-length $\hat\gamma:[0,L]\to\R^2$. Being given the
  ansatz for $\hat\gamma$ from \eqref{Gl ansatz gamma} (cf. Proposition \ref{Prop Navier}) one finds that
  solving~\eqref{Gl Dirichlet BVP} is equivalent to finding admissible parameters $a,b,L$ such that
  \begin{equation} \label{Gl Dirichlet BC}
      \vecII{\cos(\theta_1 + \int_0^L \kappa_{a,b}(r)\,dr)}{\sin(\theta_1+ \int_0^L \kappa_{a,b}(r)\,dr)}
      = \vecII{\cos(\theta_2)}{\sin(\theta_2)}
      \;\text{and}\;
      \vecII{\int_0^L \cos(\theta_1+\int_0^t \kappa_{a,b}(r)\,dr)\,dt}{\int_0^L \sin(\theta_1+\int_0^t
      \kappa_{a,b}(r)\,dr)\,dt} = B-A
  \end{equation}
  where we may without loss of generality assume $|\theta_2-\theta_1|\leq \pi$. In order to formulate our
  result we introduce the functions $\bar\alpha_j,\bar\beta$ and $\bar z\in [-1,1]$ as follows:
  \begin{align} \label{Gl Dirichlet Definition hatbeta hatalpha}
    \begin{aligned}
    \bar\beta(z;\eta) &:= z-\bar z,\qquad\text{where }
    \bar z:=\eta\cos(\theta_2-\theta_1+\sigma_1\arccos(z^2))^{1/2},  \\
    \bar\alpha_j(z;\eta) &:= \Big(j+\Big\lceil
    \frac{\sigma_1\cn^{-1}(z)-\sigma_2\cn^{-1}(\bar z)}{T}\Big\rceil\Big) C - \sigma_1\int_z^1
    \frac{t^2}{\sqrt{1-t^4}}\,dt + \sigma_2\int_{\bar z}^1 \frac{t^2}{\sqrt{1-t^4}}\,dt
    \end{aligned}
  \end{align}
  where $\eta\in\{-1,+1\},j\in\N_0$ and $z\in [-1,1]$ is chosen in such a way
  that $\bar z\in [-1,1]$ is well-defined. Moreover we will use the abbreviation
  $v:=R(-\theta_1)(B-A)\in\R^2$ or equivalently $v=(v_1,v_2)$ with
  \begin{align} \label{Gl Dirichlet Definition v}
    \begin{aligned}
    v_1 &= \quad\cos(\theta_1)(B_1-A_1)+\sin(\theta_1)(B_2-A_2),\\
    v_2 &= -\sin(\theta_1)(B_1-A_1)+\cos(\theta_1)(B_2-A_2).
    \end{aligned}
  \end{align}
  The solution theory for the Dirichlet problem then reads as follows:

  \begin{thm} \label{Thm Dirichlet}
    Let $A,B\in\R^2,A\neq B$ and $\theta_1,\theta_2\in\R$ be given such that $|\theta_2-\theta_1|\leq \pi$
    holds, let $\sigma_1,\sigma_2\in\{-1,+1\},j\in\N_0$.  Then all solutions $\gamma\in C^\infty([0,1],\R^2)$
    of the Dirichlet problem \eqref{Gl Dirichlet BVP} which are regular curves of type $(\sigma_1,\sigma_2,j)$
    are given by \eqref{Gl ansatz gamma} for
    \begin{align} \label{Gl Dirichlet abL}
      \begin{aligned}
      a&= \frac{\sqrt{2}(\bar\alpha_j(z;\eta)^2+\bar\beta(z;\eta)^2)^{1/2}}{\|B-A\|},\\
      b&= \sigma_1 \cn^{-1}(z),\\
      L &= \frac{\|B-A\| \big(\big(j+\big\lceil \frac{\sigma_1\cn^{-1}(z)-\sigma_2\cn^{-1}(\bar
       z)}{T}\big\rceil\big) T-\sigma_1\cn^{-1}(z) +\sigma_2\cn^{-1}(\bar z)
        \big)}{\sqrt{2}(\bar\alpha_j(z;\eta)^2+  \bar\beta(z;\eta)^2)^{1/2}}
      \end{aligned}
    \end{align}
	where $z\in [-1,1]$ is any solution of
	\begin{align}\label{Gl Dirichlet equation for z}
	  \begin{aligned}
      &\bar\beta(z;\eta)(z^2 v_1- \sigma_1 \sqrt{1-z^4}v_2)
      = \bar\alpha_j(z;\eta)(z^2v_2+\sigma_1\sqrt{1-z^4}v_1) \\
      &\text{provided }\qquad
      0\leq \sigma_2(\theta_2-\theta_1+\sigma_1\arccos(z^2))\leq \pi/2 \\
      &\text{and}\quad
      z=\pm 1\Rightarrow\sigma_1= \pm 1,\quad \bar z=\pm 1\Rightarrow\sigma_2= \pm 1,\quad
      z=\bar z,\sigma_1=\sigma_2 \Rightarrow j\geq 1.
      \end{aligned}
    \end{align}
	Here, $\bar\alpha_j,\bar\beta,\bar z,v$ are given by \eqref{Gl Dirichlet Definition hatbeta hatalpha} and
    \eqref{Gl Dirichlet Definition v}.
  \end{thm}
  \begin{proof}
    As in the proof of Theorem \ref{Thm Navier BVP the general case} we only show that a
    $(\sigma_1,\sigma_2,j)$-type solution of \eqref{Gl Dirichlet BVP} generated by
    $a,b,L$ satisfies the conditions stated above, i.e. we
    will not check that these conditions are indeed sufficient. By Proposition \ref{Prop Navier} (ii) we may
    write
    \begin{align} \label{Gl Dirichlet parameters}
      \begin{aligned}
      b+aL=mT+\bar b\quad\text{for } m=j+ \lceil (b-\bar b)/T\rceil \in\N_0 \text{ and } \\
      b,\bar b \in [-T/2,T/2) \quad\text{such that}\quad
      \sign(b)=\sigma_1,\;\sign(\bar b)=\sigma_2.
      \end{aligned}
    \end{align}
    In the following we set $z:= \cn(b),\bar z:=\cn(\bar b)$ so that our conventions
    $\sign(0)=1$ and $b,\bar b\in [-T/2,T/2)$ and $L>0$ imply the third line in
    \eqref{Gl Dirichlet equation for z}.

    \medskip

    {\it 1st step:}\; Using the first set of equations in  \eqref{Gl Dirichlet BC} we prove the second
    line in \eqref{Gl Dirichlet equation for z} and
    \begin{align} \label{Gl Dirichlet cn(tildeb)}
      \bar z
      =  \eta \cos(\theta_2-\theta_1+\sigma_1\arccos(z^2))^{1/2}
       \quad\text{for some }\eta\in\{-1,+1\}.
    \end{align}
    Indeed, \eqref{Gl Dirichlet BC} implies that there is some $l\in\Z$ such that
    \begin{align} \label{Gl Dirichlet angles}
      \begin{aligned}
      \theta_2-\theta_1 +2l\pi
      &= \int_0^L \kappa_{a,b}(r)\,dr \\
      &=  \int_b^{b+aL} \sqrt{2}\cn(r)\,dr \\
      &\stackrel{\eqref{Gl Dirichlet parameters}}{=}  m\int_0^T \sqrt{2}\cn(r)\,dr + \int_b^{\bar b}
      \sqrt{2}\cn(r)\,dr \\
      &=  0 -\sigma_1\int_0^{|b|} \sqrt{2}\cn(r)\,dr + \sigma_2 \int_0^{|\bar b|} \sqrt{2}\cn(r)\,dr   \\
      &= -\sigma_1\arccos(\cn^2(b))+\sigma_2\arccos(\cn^2(\bar b)) \\
      &= -\sigma_1\arccos(z^2)+\sigma_2\arccos(\bar z^2).
      \end{aligned}
    \end{align}
    (Since we have $|\theta_2-\theta_1|\leq \pi$ 
    and since the right hand side in \eqref{Gl Dirichlet angles} lies in $[-\pi/2,\pi/2]$
    we even know $l=0$.) Rearranging the above equation and applying sine respectively cosine gives
    \begin{align*}
      \sigma_2 \sqrt{1-\bar z^4}
      &= \sin(\sigma_2\arccos(\bar z^2))
      = \sin(\theta_2-\theta_1+\sigma_1\arccos(z^2)), \\
      \bar z^2
      &= \cos(\sigma_2\arccos(\bar z^2))
      = \cos( \theta_2-\theta_1+\sigma_1\arccos(z^2))
    \end{align*}
    which proves the second line in \eqref{Gl Dirichlet equation for z} and \eqref{Gl Dirichlet cn(tildeb)}.

    \medskip

    {\it 2nd step:}\; Next we use the second set of equations in \eqref{Gl Dirichlet BC}
    to derive the first line in \eqref{Gl Dirichlet equation for z}  and the formulas for $a,b,L$ from
    \eqref{Gl Dirichlet abL}. Using the addition theorems for sine and cosine we find
    \begin{align} \label{Gl Dirichlet v}
      \begin{aligned}
      v
      &= R(\theta_1)^T(B-A) \\
      &\stackrel{\eqref{Gl Dirichlet BC}}{=}
        R(\theta_1)^T \vecII{\int_0^L \cos(\theta_1+\int_0^t \kappa_{a,b}(r)\,dr)\,dt}{\int_0^L\sin(\theta_1+
        \int_0^t \kappa_{a,b}(r)\,dr)\,dt} \\
      &= R(\theta_1)^T R(\theta_1) \vecII{\int_0^L \cos(\int_0^t \kappa_{a,b}(r)\,dr)\,dt}{\int_0^L
      \sin(\int_0^t \kappa_{a,b}(r)\,dr)\,dt} \\
      &\stackrel{\eqref{Gl metaThm 2}}{=}
  	  \matII{\cn^2(b)}{\sqrt{2}\sn(b)\dn(b)}{-\sqrt{2}\sn(b)\dn(b)}{\cn^2(b)} \vecII{\alpha}{\beta} \\
  	  &=  \matII{z^2}{\sigma_1\sqrt{1-z^4}}{-\sigma_1\sqrt{1-z^4}}{z^2} \vecII{\alpha}{\beta}
	  \end{aligned}
    \end{align}
    where $\alpha,\beta$ were defined in \eqref{Gl metaThm 3}. From \eqref{Gl Dirichlet v} we obtain
    the equations
    \begin{equation} \label{Gl Dirichlet alphabeta}
      \alpha = z^2v_1 -\sigma_1\sqrt{1-z^4}v_2,\qquad
      \beta =  \sigma_1\sqrt{1-z^4}v_1 + z^2v_2.
    \end{equation}
    Straightforward calculations give
    \begin{align*}
      \beta
      &= \frac{\sqrt{2}}{a}\int_b^{b+aL}\sn(t)\dn(t)\,dt
      = \frac{\sqrt{2}}{a}(\cn(b)-\cn(b+aL))
      = \frac{\sqrt{2}}{a}(z-\bar z)
      = \frac{\sqrt{2}}{a}\bar\beta(z;\eta)
    \intertext{and}
      \alpha
      &= \frac{1}{a}\int_b^{b+aL}\cn^2(t)\,dt
       \;=\; \frac{1}{a}\int_b^{mT+\bar b}\cn^2(t)\,dt \\
      &= \frac{1}{a}\Big( m\int_0^T\cn^2(t)\,dt - \sigma_1 \int_0^{|b|}\cn^2(t)\,dt + \sigma_2
       \int_0^{|\bar b|}\cn^2(t)\,dt \Big)   \\
      &= \frac{\sqrt{2}}{a}\Big(mC
          - \sigma_1\int_{z}^1 \frac{t^2}{\sqrt{1-t^4}}\,dt
          + \sigma_2\int_{\bar z}^1 \frac{t^2}{\sqrt{1-t^4}}\,dt
         \Big)\\
      &= \frac{\sqrt{2}}{a} \bar\alpha_j(z;\eta)
    \end{align*}
    where we have used $m=j+ \lceil (b-\bar b)/T \rceil$ as well as
    $b=\sigma_1\cn^{-1}(z),\bar b=\cn^{-1}(\bar z)$ in order to derive the last equality. Using
    $\alpha^2+\beta^2=\|v\|^2=\|B-A\|^2$ from \eqref{Gl Dirichlet v} we get
    the formula for $a$ from \eqref{Gl Dirichlet abL} as well as
    $$
      \bar\beta(z;\eta)(z^2 v_1- \sigma_1 \sqrt{1-z^4}v_2)
      \stackrel{\eqref{Gl Dirichlet alphabeta}}{=} \frac{a}{\sqrt{2}}\cdot \alpha\beta
      \stackrel{\eqref{Gl Dirichlet alphabeta}}{=}  \bar\alpha_j(z;\eta)(z^2v_2+\sigma_1\sqrt{1-z^4}v_1)
    $$
    so that the first line in \eqref{Gl Dirichlet equation for z} holds true. Finally, the formula $b=
    \sigma_1\cn^{-1}(z)$ immediately follows from $\cn(b)=z,\sign(b)=\sigma_1$ and the formula for $L$ is a direct
    consequence of the formulas for $a,b,\bar b$ and $aL=mT+\bar b-b$, see \eqref{Gl Dirichlet parameters}.
  \end{proof}

  Let us finally show how the results of Deckelnick and Grunau \cite{DecGru_Boundary_value_problems} can
  be reproduced by Theorem~\ref{Thm Dirichlet}. In  \cite{DecGru_Boundary_value_problems} the authors analysed
  the Dirichlet problem for graph-type curves given by $\gamma(t)=(t,u(t))$
  where $u:[0,1]\to\R$ is a smooth positive symmetric function. The Dirichlet boundary conditions were given
  by $u(0)=u(1)=0,u'(0)=-u'(1)=\beta$ for $\beta>0$ which in our setting corresponds to $A=(0,0),B=(1,0)$ and
  $\theta_1=-\theta_2=\arctan(\beta)>0$.
  They showed that for any given $\beta>0$ this boundary value problem has precisely one symmetric positive
  solution $u$. This result can be derived from Theorem \ref{Thm Dirichlet} in the following way.
  Looking for graph-type solutions of \eqref{Gl Dirichlet BVP} requires to restrict the attention to
  $(\sigma_1,\sigma_2,j)$-type solutions with $j=0$.  The symmetry condition and Corollary~\ref{Cor
  symmetry}~(i) imply that $aL+2b$ is a multiple of~$T$ so that
  $\sigma_1\cn^{-1}(z)+\sigma_2\cn^{-1}(\bar z)$ has to be a multiple of $T$.
  From this and $\theta_2-\theta_1=-2\theta_1<0$ one can rule out the case $\sigma_1=\sigma_2$ and further
  reductions yield ${\sigma_1=1},{\sigma_2=-1},z=\bar z= \eta \cos(\theta_1)^{1/2}$ for some
  $\eta\in\{-1,+1\}$.
  Finally one can check that $\eta=1$ produces a symmetric solution which is not of graph-type in the
  sense of \cite{DecGru_Boundary_value_problems} since the slope becomes $\pm\infty$ at two points and thus
  $u\notin C^2((0,1),\R^2)$. Hence, $\eta$ must be -1 so that we end up with
  $$
    a = \sqrt{2}\Big(C-2\int_{z}^1 \frac{t^2}{\sqrt{1-t^4}}\,dt\Big),\qquad
    b= \cn^{-1}(z),\qquad
    L = \frac{T-2\cn^{-1}(z)}{\sqrt{2}(C-2\int_{z}^1\frac{t^2}{\sqrt{1-t^4}}\,dt)}
  $$
   for $z=-\cos(\theta_1)^{1/2}$
  and one may verify that this choice for $a,b,L$ indeed solves \eqref{Gl Dirichlet equation for z} and
  generates the solution found in \cite{DecGru_Boundary_value_problems}. The following Corollary shows that
  this particular axially symmetric solution is only one of infinitely many other axially symmetric solutions.
  Moreover we prove that axially non-symmetric solutions exist.

  \begin{cor}
    Let $A=(0,0),B=(1,0)$ and $\theta_1=-\theta_2\in (0,\pi/2)$. Then the Dirichlet problem~\eqref{Gl
    Dirichlet BVP} has infinitely many axially symmetric solutions $(\gamma_{1,j})$ and infinitely many
    axially non-symmetric solutions $(\gamma_{2,j})$ satisfying $W(\gamma_{1,j}),W(\gamma_{2,j})\to\infty$ as
    $j\to\infty$.
  \end{cor}
  \begin{proof}
    The existence of infinitely many axially symmetric solutions $(\gamma_{1,j})_{j\in\N_0}$ of \eqref{Gl
    Dirichlet BVP} follows from Theorem \ref{Thm Dirichlet}, Corollary \ref{Cor symmetry} and the fact that
    for every $\eta\in\{-1,+1\}$ and every $j\in\N_0$ a solution of \eqref{Gl Dirichlet equation for z} is
    given by $\sigma_1=1,\sigma_2=-1,z=\bar z=\eta\cos(\theta_1)^{1/2}$. The formula for $a$ and $W$ from
    Theorem \ref{Thm Dirichlet} implies that these solutions satisfy $W(\gamma_{1,j})\to\infty$ as
    $j\to\infty$ and $\sigma_1\cn^{-1}(z)+\sigma_2\cn^{-1}(\bar z)=0$ implies that $\gamma_{1,j}$ is axially
    symmetric.
    
    \medskip
       
    The existence of infinitely many axially non-symmetric solutions $(\gamma_{2,j})$ may be proved for
    parameters $\sigma_1=1,\sigma_2=-1,\eta=1$ and all $j\in\N_0$ satisfying $j\geq j_0$ where $j_0\in\N$ is
    chosen such that the inequality
    $$
      \bar\beta(-1;1)\cos(\theta_1)+\sin(\theta_1)\bar\alpha_{j_0}(-1;1) >0
    $$
    holds. Indeed, using $v_1=\cos(\theta_1),v_2=-\sin(\theta_1)$ in the setting of Theorem \ref{Thm
    Dirichlet} it suffices to find a zero of the function
    $$
      f_j(z)
      := \bar\beta(z;1)(z^2 \cos(\theta_1)+ \sqrt{1-z^4}\sin(\theta_1))
      -\bar\alpha_j(z;1)(-z^2\sin(\theta_1)+\sqrt{1-z^4}\cos(\theta_1))
    $$
    in $(-1,1)$. To this end we apply the intermediate value theorem to $f_j$ on the interval
    $(-1,-z^*)$ where $z^*:=\cos(\theta_1)^{1/2}$. Calculations and $j\geq j_0$ imply
    $$
      f_j(-z^*)=-2\cos(\theta_1)^{1/2}<0,\qquad
      f_j(-1) = \bar\beta(-1;1)\cos(\theta_1)+\bar\alpha_j(-1;1)\sin(\theta_1)   >0
    $$
    so that the intermediate value theorem provides the existence of at least one zero in this interval.
    Hence, $-1<z<-z^*<0<\bar z<1$ implies that $\sigma_1 \cn^{-1}(z)+\sigma_2\cn^{-1}(\bar z)$ is not a
    multiple of $T$ proving that the constructed solutions are not axially symmetric.
  \end{proof}

 \section*{Acknowledgements}
 
 The work on this project was financially supported by the Deutsche Forschungsgemeinschaft (DFG, German
 Research Foundation) - grant number MA 6290/2-1.

\bibliographystyle{plain}

\begin{thebibliography}{99}  
  \bibitem{AbrSte_Handbook} M.Abramowitz, I.A. Stegun: {\it Handbook of mathematical functions with formulas,
  graphs, and mathematical tables}, National Bureau of Standards Applied Mathematics Series 55 (1964).
  \bibitem{DecGru_Boundary_value_problems} K. Deckelnick, H.-C. Grunau: {\it Boundary value problems for the
  one-dimensional Willmore equation}, Calc. Var. Partial Differential Equations 30 (2007), no.~3, 293--314.
  \bibitem{DecGru_Stability_and_symmetry} K. Deckelnick, H.-C. Grunau: {\it Stability and symmetry in the
  Navier problem for the one-dimensional Willmore equation}, SIAM J. Math. Anal. 40 (2008/09), no.~5,
  2055--2076.
  \bibitem{Koo_ideal_bent_contours} C. Koos, C.-G. Poulton, L. Zimmerman, L. Jacome, J. Leuthold, W. Freude:
  {\it Ideal bend contour trajectories for single-mode operation of low-loss overmoded waveguides}, Photonics
  Technology Letters, IEEE, 2007, 19 (11), pp. 819 - 821. 
  \bibitem{Lin_Explicit_elastic} A. Linn{\'e}r: {\it Explicit elastic curves}, Ann. Global Anal. Geom. 16
  (1998), no.~5, 445--475.
  \bibitem{Sch_The_Willmore_boundary} R. Sch{\"a}tzle: {\it The {W}illmore boundary problem},
    Calc. Var. Partial Differential Equations 37 (2010), no.~3-4, 275--302.
  \bibitem{Sin_Lectures} D.A. Singer: {\it Lectures on elastic curves and rods},
    Curvature and variational modeling in physics and biophysics,
    AIP Conf. Proc. 1002 (2008), 3--32. 		
\end{thebibliography}

\end{document}